\documentclass[12pt,leqno,a4paper]{amsart}

\usepackage{amssymb,enumerate}
\usepackage{amsmath,amscd,amsthm}
\usepackage{mathrsfs}
\usepackage{txfonts}
\usepackage{hyperref,cite}
\usepackage{xcolor}
\usepackage{framed}

\newcommand{\Syl}{\operatorname{Syl}}
\newcommand{\Aut}{\operatorname{Aut}}
\newcommand{\Out}{\operatorname{Out}}
\newcommand{\Irr}{\operatorname{Irr}}
\newcommand{\IBr}{\operatorname{IBr}}
\newcommand{\IRRK}{\operatorname{IRR}_K}
\newcommand{\CHARK}{\operatorname{CHAR}_K}

\newcommand{\Cab}{\operatorname{Cab}}
\newcommand{\Res}{\operatorname{Res}}
\newcommand{\Ind}{\operatorname{Ind}}
\newcommand{\R}{\operatorname{R}}
\newcommand{\C}{\operatorname{C}}
\newcommand{\N}{\operatorname{N}}
\newcommand{\Z}{\operatorname{Z}}

\newcommand{\bH}{\mathbf{H}}
\newcommand{\bG}{\mathbf{G}}
\newcommand{\bL}{\mathbf{L}}
\newcommand{\bT}{\mathbf{T}}

\newcommand{\bS}{\mathbf{S}}
\newcommand{\bP}{\mathbf{P}}

\newcommand{\FF}{\mathbb{F}}
\newcommand{\ZZ}{\mathbb{Z}}
\newcommand{\CC}{\mathbb{C}}

\newcommand{\tbG}{\widetilde{\mathbf{G}}}
\newcommand{\tbS}{\widetilde{\mathbf{S}}}
\newcommand{\tbL}{\widetilde{\mathbf{L}}}

\newcommand{\tG}{\widetilde{G}}
\newcommand{\ts}{\tilde{s}}

\newcommand{\SL}{\operatorname{SL}}

\newcommand{\GL}{\operatorname{GL}}

\newcommand{\cB}{\mathcal{B}}

\newcommand{\cG}{\mathcal{G}}

\newcommand{\cP}{\mathcal{P}}
\newcommand{\cQ}{\mathcal{Q}}

\newcommand{\fI}{\mathfrak{I}}
\newcommand{\fS}{\mathfrak{S}}
\newcommand{\fA}{\mathfrak{A}}

\newcommand{\cE}{\mathscr{E}}

\newcommand{\sJ}{\mathscr{J}}

\newcommand{\ty}[1]{\mathsf{#1}}

\let\embed=\hookrightarrow

\let\la=\lambda

\let\ti=\times
\let\al=\alpha

\theoremstyle{theorem}

\newtheorem{thm}{Theorem}[section]
\newtheorem{lem}[thm]{Lemma}
\newtheorem{prop}[thm]{Proposition}
\newtheorem{cor}[thm]{Corollary}
\newtheorem{question}[thm]{Question}

\newtheorem{cond}[thm]{Condition}

\newtheorem*{conjA}{Conjecture A}
\newtheorem*{thmB}{Theorem B}

\theoremstyle{definition}
\newtheorem{defn}[thm]{Definition}

\begin{document}
	
%%%%%%%%%%%%%%%%%%%%%%%%%%%%%%%%%%%%%%%%%%%%%%%%%%%%%%%%%%%%
\title[]{On the number of $\ell$-regular conjugacy classes}
%%%%%%%%%%%%%%%%%%%%%%%%%%%%%%%%%%%%%%%%%%%%%%%%%%%%%%%%%%%%
	
\author{Zhicheng Feng}
\address[Z. Feng]{Shenzhen International Center for Mathematics and Department of Mathematics, Southern University of Science and Technology, Shenzhen 518055, China}
\makeatletter
\email{fengzc@sustech.edu.cn}
\makeatother

\dedicatory{Dedicated to Michel Brou\'e on the occasion of his 80th birthday}	
	
\thanks{The first author gratefully acknowledges financial support by NSFC (12622101), the second author by NSFC (12431001).}

\begin{abstract}
In this paper, we prove that for $\ell>5$, the number of $\ell$-regular conjugacy classes of a finite group is at least that of the normalizer of a Sylow $\ell$-subgroup.
\end{abstract}

	
	
\date{\today}
	
\maketitle

%%%%%%%%%%%%%%%%%%%%%%%%%%%%%%%%%%%%%%%%%%%%%%%%%%%%%%%%%%%%%%%%%%%%%%%%%
\section{Introduction}

Alperin weight conjecture (cf. \cite{Al87}) is one of the central problems in modular representation theory of finite groups.
It remains open now.
In 2011, Navarro and Tiep \cite{NT11} reduced it to simple groups; they introduced the inductive Alperin weight condition, and proved that if all non-abelian simple groups satisfy this condition, then Alperin weight conjecture is true.
After the work of \cite{AD12,Ma14,Sc16,FLZ21,Li21,FM22,FLZ22,FLZ23,FLZ26,FYZ23,AHL24}, except for a few simple groups with exceptional Schur multipliers, the inductive blockwise Alperin weight condition remains open only for groups of types $\ty D$ and $\ty E$.
However, as pointed by Malle, Navarro and Tiep \cite[p.1206]{MNT23}, it is fair to say that a full proof of AWC seems yet, unfortunately, out of reach.

Alperin weight conjecture has the following nice consequence.

\begin{conjA}[Alperin {\cite[Conjecture~A]{MNT23}}]\label{conj-bound}
If $G$ is a finite group, $\ell $ is a prime, and $P$ is a Sylow $\ell$-subgroup of $G$, then
\[	l(G) \ge l(\N_G(P)). \]
Here $l(G)$ denotes the number of $\ell$-regular conjugacy classes of $G$.
\end{conjA}	
	
In 2023, Malle, Navarro and Tiep \cite{MNT23} reduced Conjecture~A to simple groups. 
They defined the inductive Sylow-AW condition, and proved that if all non-abelian simple groups satisfy the inductive Sylow-AW condition, then Conjecture~A holds.
They also prove Conjecture~A for the prime $\ell=2$.
In 2026, Feng, Mart\'{i}nez and Rossi \cite{FMR26} developed this reduction, and established a self-reduction theorem.
	
In this paper, we prove Conjecture~A for all primes greater than 5.

\begin{thmB}
Assume that $\ell>5$. Then Conjecture~A holds.
\end{thmB}
	
Therefore, combining a result of Malle, Navarro and Tiep \cite[Thm.~B]{MNT23}, Conjecture~A is now open only for the primes 3 and 5.

In the proof of the inductive McKay conjecture, Malle's group, which is the normalizer of a Sylow $e$-torus and introduced by Malle \cite{Ma07} in 2007, is a crucial ingredient (see for instance \cite{CS26}).
In this paper, we use Malle's groups instead of normalizers of Sylow $\ell$-subgroup, and thus use the methods developed in the verification of the inductive McKay conjecture.	
This leave the inductive Sylow-AW condition open only for simple groups of types $\ty E$ and odd bad primes (see Theorem~\ref{thm:exc-covering}, Theorem~\ref{thm:lie-good} and Corollary~\ref{cor:classical})

This paper is organized as follows. 
After introducing some notation and elementary results in Section ~\ref{sec:pre}, we recall the central isomorphism between character triples and the inductive Sylow-AW condition in Section~\ref{sec:central}.
In Section ~\ref{sec:exceptional}, we deal with the exceptional covering groups.
Finally, in Section~\ref{sec:lie} we verify the inductive Sylow-AW condition for groups of Lie type at good primes.

%%%%%%%%%%%%%%%%%%%%%%%%%%%%%%%%%%%%%%%%%%%%%%%%%%%%%%%%%%%%%%%%%%%%%%%%%
\section{Preliminaries}\label{sec:pre}
	
\subsection{Notation}
	
For any prime power $r$, we denote by $\mathbb F_r$ the field of $r$ elements and by $\overline{\mathbb F}_r$ its algebraic closure.

In this paper, we consider modular representations with respect to the prime $\ell$. 
For the character theory of a finite group $G$, we mainly follow the notation of \cite{Is06}, \cite{Na98} and \cite{Na18}.
As usual, $\Irr(G)$ denotes the set of irreducible complex characters of $G$, while $\IBr(G)$ denotes the set of irreducible ($\ell$-)Brauer characters of $G$.
Let $K\in\{\overline{\mathbb F}_\ell,\, \CC\}$.
We write $\IRRK = \Irr$ if $K = \CC$, and $\IRRK = \IBr$ if $K = \overline{\mathbb F}_\ell$.
For instance, we have $\IRRK(G) = \Irr(G)$ when $K = \CC$, and $\IRRK(G) = \IBr(G)$ when $K = \overline{\mathbb F}_\ell$.

The restriction of a (Brauer) character $\chi$ of $G$ to a subgroup $H \le G$ is denoted by $\Res^G_H(\chi)$, and we denote by $\IRRK(H \mid \chi)$ the set of irreducible constituents of $\Res^G_H(\chi)$.
For a (Brauer) character $\psi$ of $H$, let $\Ind^G_H(\psi)$ be the character induced from $\psi$ to $G$, and denote by $\IRRK(G \mid \psi)$ the set of irreducible constituents of $\Ind^G_H(\psi)$.
For a subset $\mathcal H\subseteq\IRRK(H)$, we define
\[\IRRK(G\mid\mathcal H)=\bigcup_{\psi\in\mathcal H}\IRRK(G\mid\psi)\]
and for a subset $\cG\subseteq\IRRK(G)$, we define
\[\IRRK(H\mid\mathcal G)=\bigcup_{\chi\in\mathcal G}\IRRK(H\mid\chi).\]
If $N \unlhd G$, we sometimes identify the (Brauer) characters of $G/N$ with those of $G$ whose kernel contains $N$.

The cardinality of a set or the order of a finite group $X$ is denoted by $|X|$.
If a group $A$ acts on a finite set $X$, we write $A_x$ or $I_A(x)$ for the stabilizer of $x \in X$ in $A$.
%If a group $B$ acts a finite group $G$ via automorphisms, then $B$ acts on $\IRRK(G)$.
%We denote by $\Irr_B(G)$ the set of $B$-invariant characters in $\Irr(G)$.

\begin{defn}[{\cite[Definition~5.7]{CS17a}}]
Let \( N \unlhd G \) and \( \mathcal{N} \subseteq \IRRK(N) \). We say that \textit{maximal extendibility holds for \(\mathcal{N}\) with respect to \( N \unlhd G \)} if every \( \chi \in \mathcal{N} \) extends (as irreducible character) to \( G_\chi \). Then, an \textit{extension map for \(\mathcal{N}\) with respect to \( N \unlhd G \)} is a map
\[
\Lambda \colon \mathcal{N} \to \bigcup_{N \leq I \leq G} \IRRK(I),
\]
such that for every \( \chi \in \mathcal{N} \), the character \( \Lambda(\chi) \in \IRRK(G_\chi) \) is an extension of \( \chi \).
\end{defn}

For a prime $r$ and any integer $m$ with $r\nmid m$, we set
\[e_r(m) := \text{order of } m \text{ modulo }
\begin{cases} 
	\ r & \text{if } r\ \text{is odd} , \\
\ 	4 & \text{if } r = 2.
\end{cases}
\]

\subsection{On maximal abelian normal subgroups}

	Let $G$ be a group and $A$ a subgroup of $G$.
	Then $G$ is said to be \emph{Cabanes} with \emph{Cabanes subgroup} $A$ if $A$ is the unique maximal abelian normal subgroup of $G$.
	In this situation, we write $A=\Cab(G)$.

See \cite[\S2]{Ca94} for some properties for Cabanes groups.
It was observed by Cabanes \cite{Ca94} that the Sylow $\ell$-subgroups of finite reductive groups are Cabanes provided that $\ell\ge 5$ and $\ell$ is not the defining prime.

\begin{lem}\label{lem:self-cen}
	Let $G$ be a group.
	\begin{enumerate}[\rm(i)]
		\item If $A$ is a normal abelian subgroup of $G$ with $\C_G(A)\subseteq A$, then $A$ is a maximal abelian normal subgroup of $G$.
		\item Assume further that $G$ is Cabanes with $A=\Cab(G)$. If $G/A$ is supersolvable, then $\C_G(A)=A$.
	\end{enumerate}	
\end{lem}

\begin{proof}
	The assertion (i) follows by construction, and
	we consider (ii). Assume that $\C_G(A) \neq A$.
   Recall that every minimal normal subgroup of a supersolvable group is cyclic.
   Since $A \unlhd \C_G(A) \unlhd G$, we have $\C_G(A)/A \unlhd G/A$, whence there exists a normal subgroup $D$ of $G$ such that $A\subsetneq D \subseteq \C_G(A)$ and $D/A$ is cyclic.
   We may write $D = \langle A, d \rangle$ for some $d \in D$.
   Then $D$ is abelian because $d \in \C_G(A)$, contradicting the maximality of $A$.
   Therefore, $\C_G(A) \subseteq A$, which implies that $\C_G(A)=A$.
\end{proof}

%%%%%%%%%%%%%%%%%%%%%%%%%%%%%%%%%%%%%%%%%%%%%%%%%%%%%%%%%%%%%%%%%%%%%%%%%%%%%%%%%%%%

\section{Central isomorphisms between character triples and the inductive conditions}\label{sec:central}

Let \( K \in \{\overline{\mathbb F}_\ell, \CC\} \). 
Let $G$ be a finite group, $H$ a subgroup of $G$.
For $\psi\in\IRRK(H)$, we denote by $\CHARK(G\mid\psi)$ the set of characters of $G$ whose constituents are in $\IRRK(G\mid\psi)$.

We say $(G,N,\theta)$ a \emph{character triple}, if $N\unlhd G$, and $\theta\in\IRRK(N)$ is $G$-invariant.
To emphasize $K$, sometimes we also call it a $K$-character triple.
We refer to Chapter 11 of \cite{Is06} and Chapter 5 and Chapter 10 of \cite{Na18} for the terminology of character triples and further details.

\subsection{Induced projective representations}

A \emph{projective $K$-representation} of $G$ is a function $\cP\colon G\to\GL_n(K)$ such that for every $g,h\in G$ there exists some $\al(g,h)\in K^\ti$ satisfying $\cP(g)\cP(h)=\al(g,h)\cP(gh)$.
Then $\al$ is a $2$-cocycle, also called the \emph{factor set} of $\cP$.

Let  $(G,N,\theta)$ be a $K$-character triple.
By \cite[Chap.~3, Thm.~5.7]{NT89}, there is a projective \(K\)-representation \(\mathcal{P} : G \to \operatorname{GL}_{\theta(1)}(K)\) of $G$, called \emph{a projective \(K\)-representation of \(G\) associated with \((G,N,\theta)\)}, satisfying that
 \(\mathcal{P}|_N\) is a linear (\(K\)-)representation of $N$ affording \(\theta\),
and \(\mathcal{P}(gn) = \mathcal{P}(g)\mathcal{P}(n)\) and \(\mathcal{P}(ng) = \mathcal{P}(n)\mathcal{P}(g)\) for every \(g \in G\) and \(n \in N\).
The factor set $\al$ of $\cP$ satisfies that $\al(1,1)=\al(g,n)=\al(n,g)=1$ for
every $n\in N$ and $g\in G$, and thus it can be regarded as a $2$-cocycle
of $G/N$.

Let $G=NH$ with $N\unlhd G$, and set $M=N\cap H$ so that $(H,M,\varphi)$ is a $K$-character triple.
Let  $\cP\colon H\to\GL_m(K)$ be a projective $K$-representation of~$H$ associated with $(H,M,\varphi)$.
Choose $n_1,\ldots,n_s\in N$ (where $s=|G|/|H|$) as representatives of the $M$-cosets in $N$.
Consequently, $n_1,\ldots,n_s$ also serve as representatives for the $H$-cosets in $G$.
For $1\le i,j\le s$, define $\widehat\cP_{i,j}\colon G\to \textrm{Mat}_{m}(K)$ by 
\[\widehat\cP_{i,j}(x) =
\begin{cases}
	\cP(n_i^{-1}xn_j)  & \ \mbox{when } n_i^{-1}xn_j\in H, \\
	0  & \ \mbox{otherwise.}
\end{cases}   \] 
Define the \emph{induced projective $K$-representation of~$\cP$ to~$G$} (with respect to $n_1,\ldots,n_s$) by
\begin{align*}
\Ind^G_{H,N}(\cP)\colon\ &G\to\GL_{ms}(K),\\ &x\mapsto
\left (
\begin{matrix}
	\widehat\cP_{1,1}(x) & \cdots & \widehat\cP_{1,s}(x) \\
	\vdots &  & \vdots \\
	\widehat\cP_{s,1}(x) & \cdots & \widehat\cP_{s,s}(x)
\end{matrix}
\right  ).
\end{align*}
Recall that this is used in the proofs of \cite[Thm.~3.14]{NS14} and \cite[Prop.~5.8]{F25}.
One shows that $\Ind^G_{H,N}(\cP)$ is a projective $K$-representation of $G$ with factor set $\widehat \al$ such that for every $x,y\in G$, $\widehat \al(x,y)=\al(x_1,y_1)$ where $x_1,y_1\in H$ satisfy $x=x_1 n_x$ and $y=y_1n_y$ for some $n_x,n_y\in N$.
Equivalently, $\widehat\al$ and $\al$ correspond via the isomorphism $G/N\cong H/M$.
If $\Ind_M^N(\varphi)\in\IRRK(N)$, then $\Ind^G_{H,N}(\cP)$ is a projective $K$-representation associated with the character triple $(G,N,\Ind_M^N(\varphi))$.
In particular, when $\cP$ is a linear $K$-representation of~$H$, the construction $\Ind^G_{H,N}(\cP)$ coincides with the usual induced linear $K$-representation of~$\cP$ to~$G$ (see e.g. \cite[p.~75]{CR62}).

\subsection{Central isomorphisms between character triples}

The notion of isomorphism between character triples is standard in representation theory and has found significant applications. A more restrictive notion, namely that of \emph{central isomorphism} of character triples, was introduced in \cite{NS14} and has proved useful in the context of reduction theorems. For further details, we refer the reader to the expositions in \cite{Na18} and \cite{Sp17}; we recall the definition as follows.

Let $(G,N,\theta)$ and $(H,M,\varphi)$ be two $K$-character triples. 
We write
\[(G,N,\theta)\succeq(H,M,\varphi)\]
if $G=NH$, $H\le G$ and $N\cap H=M$, and there exist projective $K$-representations $\cP$ and $\cP'$ associated with $(G,N,\theta)$ and $(H,M,\varphi)$ respectively, with factor sets $\al$ and
	$\al'$ such that $\al|_{H\ti H}=\al'$.
In this situation, we also say that \emph{$(G,N,\theta)\succeq(H,M,\varphi)$ is
	given by $(\cP,\cP')$}.

\begin{lem}\label{lem:iso-triple}
Let $(G,N,\theta)$ and $(H,M,\varphi)$ be two $K$-character triples such that
\[(G, N, \theta) \succeq (H, M, \varphi)\] is given by \((\mathcal{P}, \mathcal{P}')\). Then for every subgroup \(J\) with \(N \leq J \leq G\), the linear map 
\[\sigma_J : \CHARK(J \mid \theta) \longrightarrow \CHARK(J \cap H \mid \varphi)\] given by
\[
	\textup{Trace}(\cQ \otimes \cP|_J) \mapsto \textup{Trace}(\cQ|_{J \cap H} \otimes \cP'|_{J \cap H}),
\]
for any projective representation $\cQ$ of $J/N$, whose factor set is inverse to the one of $\cP|_{J}$, is a well-defined map satisfying that for any $N\le K\le J\le G$ and any $\chi\in\CHARK(J\mid\theta)$,
\begin{enumerate}[\rm(i)]
\item \(\sigma_J(\IRRK(J \mid \theta)) = \IRRK(J \cap H \mid \varphi)\),
\item $\Res^{J\cap H}_{K\cap H}(\sigma_{J}(\chi))
=\sigma_{K}(\Res^{J}_{K}(\chi))$, 
\item $\sigma_{J}(\chi\beta)=\sigma_{J}(\chi)\Res^J_{J\cap H}(\beta)$ for
every $\beta\in\CHARK(J/N)$, and
\item $\sigma_{J}(\chi)^h=\sigma_{J^h}(\chi^h)$ for every $h\in H$.
\end{enumerate}	
In particular, $\sigma$ is a strong isomorphism between the character triples $(G,N,\theta)$ and $(H,M,\varphi)$ in the sense of \cite[Problem 11.13]{Is06}.
\end{lem}

We note that here for characteristic $\ell$ case, $\textup{Trace}$ means lifting to the complex number field to get Brauer characters.

\begin{proof}
See \cite[Thm.~2.2 and Cor.~2.4]{Sp18} and \cite[Thm.~3.1]{SV16}.
\end{proof}

Let $\cP$ be a projective $K$-representation associated with a $K$-character triple $(G,N,\theta)$. Then the matrix $\cP(c)$ is scalar for all $c\in\C_G(N)$.

\begin{defn}[{\cite[Definition~3.1]{Sp17}}]
Let $(G,N,\theta)$ and $(H,M,\varphi)$ be two $K$-character triples such that
\[(G, N, \theta) \succeq (H, M, \varphi)\] is given by \((\mathcal{P}, \mathcal{P}')\).
	We write
	\[(G,N,\theta)\succeq_{c}(H,M,\varphi)\]
	if the following conditions are satisfied:
	\begin{enumerate}[\rm(i)]
		\item $\C_G(N)\le H$.
		\item For every $c\in\C_G(N)$ the scalars associated to $\cP(c)$ and
		$\cP'(c)$ coincide.
	\end{enumerate}
	In this situation, we say that \emph{$(G,N,\theta)\succeq_{c}(H,M,\varphi)$
		is given by $(\cP,\cP')$}.
\end{defn}

The relation $\succeq_c$ is called a \emph{central isomorphism} between character triples.

\subsection{A Clifford property}

\begin{prop}\label{prop:Cli}
	Let \( G \) be a finite group. Suppose that \( N \unlhd G \) and \( H \leq G \) such that \( G = NH \).  Let \( M := N \cap H \). Suppose that $\mathcal N\subseteq\IRRK(N)$ and $\mathcal M\subseteq \IRRK(M)$ are $H$-stable sets such that there exists an \( H \)-equivariant bijection
	\[	\Omega : \mathcal N \to  \mathcal M	\]
	such that \( (G_{\theta}, N, \theta) \succeq_c (H_\varphi, M, \varphi) \) for every $\theta\in\mathcal N$ and $\varphi=\Omega(\theta)$.
	Furthermore, let \( J \unlhd G \) with \( N \le J \), and let $L=J\cap H$.  Then there exists an \( H \)-equivariant bijection
	\[	\Psi :   \IRRK(J\mid\mathcal N)\to	\IRRK(L\mid\mathcal M)\]
	such that \( (G_{\vartheta}, J, \vartheta) \succeq_c (H_\phi, L, \phi) \) for every 
	$\vartheta\in\IRRK(J\mid\mathcal N)$ and $\phi=\Psi(\vartheta)$.
	
	If, moreover, $J/N$ is abelian, then $\Psi$ can be chosen to satisfy that $\Psi(\vartheta\delta)=\Psi(\vartheta)\Res^J_L(\delta)$ for every \( \vartheta \in \IRRK(J\mid\mathcal N)\) and $\delta\in\IRRK(J/N)$.
\end{prop}

\begin{proof}
	Let $\mathscr T$ be a complete set of representatives of the $G$-orbits in $\mathcal N$. Then $\Omega(\mathscr T)$ is a complete set of representatives of the $H$-orbits in $\mathcal M$.
	
	Let $\theta\in\mathscr T$ and $\varphi=\Omega(\theta)$.
	Suppose that \( (G_{\theta}, N, \theta) \succeq_c (H_\varphi, M, \varphi) \) is given by $(\cP,\cP')$.
	Since $\Omega$ is $H$-equivariant, $G_\theta=NH_{\varphi}$ and $\N_G(J_\theta)=N\N_H(L_\varphi)$.
	Let $\sigma$ be the strong isomorphism between the character triples $(G_\theta,N,\theta)$ and $(H_\varphi,M,\varphi)$ as in Lemma~\ref{lem:iso-triple}.
	Denote by \[\widehat\Omega_\theta\colon\IRRK(J_\theta\mid\theta)\to\IRRK(L_\varphi\mid\varphi)\]
	the restriction of $\sigma_{J_\theta}$.
	Then $\widehat\Omega_\theta(\widehat\theta\beta)=\widehat\Omega_\theta(\widehat\theta)\Res^{J_\theta}_{L_\varphi}(\beta)$ for every $\widehat\theta\in\IRRK(J_\theta\mid\theta)$ and $\beta\in\IRRK(J_\theta/N)$ with $\beta(1)=1$, and $\widehat\Omega_\theta$ is $\N_H(L_\varphi)_\varphi$-equivariant.
	
	Let $\widehat\theta\in\IRRK(J_\theta\mid\theta)$ and $\widehat\varphi=\widehat\Omega_\theta(\widehat\theta)$.
	Suppose that $\widehat\theta=\textup{Trace}(\cP|_{J_\theta}\otimes\cQ)$ and $\widehat\varphi=\textup{Trace}(\cP'|_{L_\varphi}\otimes\cQ)$ where $\cQ$ is a projective representation of $J_\theta/N\cong L_\varphi/M$.
	Let $\widehat\cP$ be a projective representation associated with $((G_\theta)_{\widehat\theta},J_\theta,\widehat\theta)$.
	Then by the proof of \cite[Chap.~3, Thm.~5.8]{NT89} $\widehat\cP=\cP|_{(G_\theta)_{\widehat\theta}}\otimes \mathcal R$ where $\mathcal R$ is a projective representation of $(G_\theta)_{\widehat\theta}/N$ such that $\mathcal R|_{J_\theta}=\cQ$.
	Now define $\widehat\cP'=\cP'|_{(H_\varphi)_{\widehat\varphi}}\otimes\mathcal R|_{(H_\varphi)_{\widehat\varphi}}$.
	Then similar as the proof of \cite[Prop.~3.9]{NS14}, one checks that $(\widehat\cP,\widehat\cP')$ gives 
	\[((G_\theta)_{\widehat\theta},J_\theta,\widehat\theta)\succeq_c ((H_\varphi)_{\widehat\varphi},L_\varphi,\widehat\varphi).\] 
	
	Let $\vartheta=\Ind^J_{J_\theta}(\widehat\theta)$ and $\phi=\Ind^L_{L_\varphi}(\widehat\varphi)$.
	By Clifford theory, $\vartheta\mapsto\phi$ yields a bijection  \[\Psi_\theta\colon\IRRK(J\mid\theta)\to\IRRK(L\mid\varphi)\]
	which is then $H_\varphi$-equivariant.
	Moreover, if $J/N$ is abelian, then for any $\delta\in\IRRK(J/N)$,
	\begin{align*}
		&\Psi_\theta(\vartheta\delta)=\Psi_\theta(\Ind^J_{J_\theta}(\widehat\theta)\delta)=
		\Psi_\theta(\Ind^J_{J_\theta}(\widehat\theta\Res^J_{J_\theta}(\delta)))\\
		=&\Ind^L_{L_\varphi}(\widetilde\Omega(\widehat\theta\Res^J_{J_\theta}(\delta)))
		=\Ind^L_{L_\varphi}(\widetilde\Omega(\widehat\theta)\Res^J_{L_\varphi}(\delta))\\
		=&\Ind^L_{L_\varphi}(\widehat\varphi)\Res^J_L(\delta)
		=\Psi_\theta(\vartheta)\Res^J_L(\delta).
	\end{align*}
	Note that $G_\vartheta=(G_\theta)_{\widehat\theta}J$ and $H_\phi=(H_\varphi)_{\widehat\varphi}L$.
	Then by the proof of \cite[Thm.~3.14]{NS14} (or \cite[Prop.~2.8]{Ro22}), $(\Ind^{G_{\vartheta}}_{(G_\theta)_{\widehat\theta},\,J}(\widehat\cP)\,,\Ind^{H_\phi}_{(H_\varphi)_{\widehat\varphi},\,L}(\widehat\cP'))$ gives
	\[ (G_{\vartheta}, J, \vartheta) \succeq_c (H_\phi, L, \phi).\]
	
	Let $\theta$ runs through $\mathscr T$.
	Then we collect the $\Psi_\theta$'s and then obtain an \( H \)-equivariant bijection
	\[	\Psi :   \IRRK(J\mid\mathcal N)\to	\IRRK(L\mid\mathcal M)\]
	which satisfies the desired property.
	This completes the proof.
\end{proof}

\subsection{The McKay conjecture}
In \cite[Thm.~4.1]{Sp17}, Sp\"ath reformulates the inductive condition of McKay conjecture in terms of central isomorphisms between character triples.
In fact, this implies a self-reduction of McKay conjecture (cf. \cite[Thm.~B]{Ro23}).
This strengthened version of the McKay conjecture has recently been established.

For a finite group $H$, we let $\Irr_{\ell'}(H)$ denote the set of irreducible characters of $H$ of degree not divisible by $\ell$.

\begin{thm}[Cabanes--Sp\"ath {\cite[Thm.~B]{CS26}}]\label{thm-MC}
Let \( G \trianglelefteq A \) be finite groups, \( \ell \) a prime and \( P \) a Sylow \( \ell \)-subgroup of \( G \). Then there exists an \( \N_A(P) \)-equivariant bijection
\[\Omega :  \Irr_{\ell'}(\N_G(P)) \to \Irr_{\ell'}(G) \]	
such that 	\[(A_{\Omega(\chi)}, G, \Omega(\chi)) \geq_c (\N_A(P)_\chi, \N_G(P), \chi),	\]	for every \( \chi \in \operatorname{Irr}_{\ell'}(\N_G(P)) \).
\end{thm}

\begin{proof}
	Based on \cite{IMN07}, Rossi \cite{Ro23} proved a self-reduction theorem of McKay conjecture, that is, this assertion holds if the inductive McKay condition from \cite[\S10]{IMN07} is satisfied for all finite non‑abelian simple groups.
	The inductive McKay condition has been verified in a series of papers   \cite{Ma08,Sp12,CS13,CS17a,CS17b,CS19,Sp23,Sp25,CS26}.
\end{proof}

\subsection{The inductive Sylow-AW condition}

The inductive condition of Conjecture~A can also be described in terms of central isomorphisms between character triples.

\begin{defn}[Malle--Navarro--Tiep {\cite[\S2]{MNT23}}]\label{defn-SAW}
	Let $\ell$ be a prime, $S$ a finite non-abelian simple group with $\ell$ dividing $|S|$, $G$ a universal $\ell'$-covering group of $S$. Let $P$ be a Sylow $\ell$-subgroup of $G$.
	We say that the \emph{inductive Sylow-AW condition holds for $S$ and $\ell$} if 
	there exists a proper subgroup $T$ of $G$ stabilised by \(\Aut(G)_{P}\) with $\N_G(P)\subseteq T$ and there exists an \(\Aut(G)_{P}\)-equivariant injection
	\[\Omega : \IBr(T) \embed \IBr(G)\] 
	satisfying that for any $\theta\in\IBr(T)$ and $\varphi=\Omega(\theta)$,
\begin{equation}\label{equ:ceniso}
(G \rtimes \Aut(G)_{\varphi}, G, \varphi) \succeq_{c} (T(G \rtimes \Aut(G))_{P,\, \theta}, T, \theta).
\addtocounter{thm}{1}\tag{\thethm}
\end{equation}	
\end{defn}	

We say that the \emph{inductive Sylow-AW condition holds for $S$} if it holds for $S$ and any prime $\ell\mid|S|$.

By \cite[Thm.~2]{MNT23}, if the inductive Sylow-AW condition is satisfied for all finite non-abelian simple groups and the prime $\ell$, then Conjecture~A holds for all finite groups for the prime $\ell$.
See also \cite[Thm.~2.13]{FMR26} for a self-reduction theorem.

Recall from \cite[p.1209]{MNT23} that the inductive Sylow-AW condition is implied by the inductive Alperin weight condition defined in \cite[\S3]{NT11}.

%%%%%%%%%%%%%%%%%%%%%%%%%%%%%%%%%%%%%%%%%%%%%%%%%%%%%%%%%%%%%%%%%%%%%%%%

\section{Exceptional covering groups}\label{sec:exceptional}

First we give a criterion for the inductive Sylow-AW condition.

\begin{thm}\label{thm:criterion-SylowAW}
Let $\ell$ be a prime, $S$ a finite non-abelian simple group, $G$ an $\ell'$-covering group of $S$. 
Let $P$ be a Sylow $\ell$-subgroup of $G$.
Assume that there exists a proper subgroup $T$ of $G$ stabilised by \(\Aut(G)_{P}\) with $\N_G(P)\subseteq T$ and the following statements are satisfied.
\begin{enumerate}[\rm(1).]
\item There exists an \(\Aut(G)_{P}\)-equivariant injection
\[\Omega : \IBr(T) \embed \IBr(G)\] 
satisfying that $\IBr(\Z(G)\mid \theta)=\IBr(\Z(G)\mid \varphi)$ for any $\theta\in\IBr(T)$ and $\varphi=\Omega(\theta)$.
\item 	For $\theta\in\IBr(T)$ and $\varphi=\Omega(\theta)$, there exists a finite group $A$ (depending only on $\varphi$) and Brauer characters $\widetilde\varphi\in\IBr(A)$ and $\widetilde\theta\in\IBr(\overline T\N_A(\overline P)_\theta)$ with the following properties.
Here we use the notation $Z=\Z(G)\cap\ker(\phi)$, $\overline G=G/Z$, $\overline P=PZ/Z$ and $\overline T=T/Z$.
\begin{enumerate}[\rm(i)]
	\item The group $A$ satisfies $\overline G\unlhd A$, $A/\C_A(\overline G)\cong \Aut(G)_\varphi$, $\C_A(\overline G)=\Z(A)$ and $\ell\nmid |\Z(A)|$.
	\item $\widetilde\varphi$ is an extension of $\overline \varphi$, and $\widetilde\theta$ is an extension of $\overline \theta$. Here $\overline \varphi\in\IBr(\overline G)$ and $\overline\theta\in\IBr(\overline T)$ are the deflations of $\varphi$ and $\theta$ respectively.
	\item $\IBr(\Z(A)\mid \widetilde\theta)=\IBr(\Z(A)\mid \widetilde\varphi)$.
\end{enumerate}	
\end{enumerate}
Then the inductive Sylow-AW condition holds for $S$ and $\ell$.
\end{thm}

\begin{proof}
By \cite[Thm.~4.3]{Sp17}, condition (2) implies (\ref{equ:ceniso}) in Definition~\ref{defn-SAW}.
\end{proof}

\begin{lem}\label{lem:S3-cri}
Let $\ell$ be a prime, $S$ a finite non-abelian simple group, $G$ an $\ell'$-covering group of $S$. 
Assume that the Sylow $t$-subgroups of $\Out(S)$ are cyclic for every prime $t$.
Let $P$ be a Sylow $\ell$-subgroup of $G$ and $T=\N_G(P)$.
Assume that there exists an \(\Aut(G)_{P,\,\nu}\)-equivariant injection
	\[\Omega : \IBr(T\mid\nu) \embed \IBr(G\mid\nu)\] 
	for any $\nu\in\Irr(\Z(G))$.
Then the inductive Sylow-AW condition holds for $S$ and $\ell$.
\end{lem}

\begin{proof}
It suffices to verify the condition (2) of Theorem~\ref{thm:criterion-SylowAW}.
Let $\theta\in\IBr(T)$ and $\varphi=\Omega(\theta)$.
We denote 
$Z=\Z(G)\cap\ker(\phi)$, $\overline G=G/Z$, $\overline P=PZ/Z$ and $\overline T=T/Z$.
Let $k$ be the splitting field of $x^{|\overline G|}-1$ over $\FF_\ell$.

Suppose that $\overline\varphi\ne 1_{\overline G}$. 
Then $\overline\varphi$ is afforded by a faithful irreducible representation $\overline G\to \GL(V)$ where $V$ is a vector space over $k$ of dimension $\varphi(1)$. 
Regard $\overline G$ as a subgroup of $\GL(V)$.
Put $A=\N_{\GL(V)}(\overline G)$.
As in the beginning of the proof of \cite[Thm.~C]{NT11}, one shows that condition (2.i) of Theorem~\ref{thm:criterion-SylowAW} holds.
In particular, $\C_A(\overline G)=\Z(A)\cong k^\ti$ is an $\ell'$-group.
Note that $V$ is also an irreducible $kA$-module, since its restriction to $\overline G$ is irreducible.
It gives an extension $\widetilde\varphi\in\IBr(A)$ of $\overline\varphi$.
Let $\gamma\in\Irr(\Z(A)\mid\widetilde\varphi)$.
Now $\overline T \Z(A)$ is a central product of $\overline T$ and $\Z(A)$ over $\overline T\cap\Z(A)=\Z(\overline G)$.
By \cite[Lemma~2.5(i)]{FS23}, $\overline \theta$ extends to $\widehat\theta\in\IBr(\overline T \Z(A))$ with $\gamma\in\Irr(\Z(A)\mid\widehat\theta)$.
By construction, $\widehat\theta$ is $\N_A(\overline P)$-invariant.
The group $\N_A(\overline P)/\overline T \Z(A)$ is isomorphic to a subgroup of $\Out(S)$, and hence its Sylow $t$-subgroups are cyclic for every prime $t$.
Therefore, $\widehat\theta$ extends to 
$\widetilde\theta\in\IBr(\N_A(P))$ by \cite[Thm.~(8.29)]{Na98}, and by construction, $\gamma\in\Irr(\Z(A)\mid\widetilde\theta)$. 

Now suppose that $\overline\varphi=1_{\overline G}$. 
Then take $A=\Aut(\overline G)$.
Thus $\C_A(\overline G)=\Z(A)=1$.
Hence it suffices to show that $\theta$ extends to $\N_A(P)$, which is implied by \cite[Thm.~(8.29)]{Na98}.
\end{proof}

\begin{lem}\label{lem:ele-lem}
Let $X$ be a finite group and $E\le\Z(X)$.	
Suppose that a group $A$ acts on $X$ such that $E$ is $A$-stable, and $A$ has a normal subgroup $D$ with $\gcd(|D|,|E|)=1$.	
Assume that $\C_E(D)=1$ and that $D$ acts trivially on $X/E$.	
Then \[X=E\ti\C_X(D),\] and $\C_X(D)\to X/E$ is an $A$-equivariant isomorphism.	
\end{lem}

\begin{proof}
By hypothesis, every $E$-coset in $X$ is $D$-stable.	
Therefore, every $E$-coset in $X$ contains a fixed point of $D$ by \cite[Thm.~3.27]{Is08}.
Let $c\in \C_X(D)$.
Then for $e\in E$, $ce\in \C_X(D)$ if and only if $e\in \C_E(D)$.
It follows from $\C_E(D)=1$ that every $E$-coset in $X$ contains exactly one fixed point of $D$.
Therefore, $\C_X(D)\to X/E$, $c\mapsto cE$ is an isomorphism.	
Moreover, $X=	E\C_X(D)$ and $E\cap \C_X(D)=1$, so $X=E\ti\C_X(D)$.
The $A$-equivariance can be checked directly.	
\end{proof}

\begin{lem}\label{lem:S3-SylAW}
Let $\ell$ be odd, and let $G$ be the universal $\ell'$-covering
of a nonabelian simple group $S$.
Suppose that $\Z(G)$ is a 2-group, and 
$\Out(S)\cong S_3$ such that $\C_{\Z(G)}(\Out(S))=1$.
For $P\in\Syl_\ell(G)$, put
$\overline P=P\Z(G)/\Z(G)$.
If $\Aut(S)_{\overline P}$ acts trivially on $\N_S(\overline P)/\overline P$, then
$\N_S(\overline P)/\overline P$ is abelian and
\[\N_G(P)/P\cong \Z(G)\times \N_S(\overline P)/\overline P.\]
For every $\nu\in\Irr(\Z(G))$, the set $\IBr(\N_G(P)\mid\nu)$ has $|\N_S(\overline P)/\overline P|$ elements, all fixed by $\Aut(G)_{P,\nu}$.

If, moreover, $\IBr(G\mid\nu)$ contains at least $|\N_S(\overline P)/\overline P|$ elements fixed by $\Aut(G)_\nu$ for every $\nu\in\Irr(\Z(G))$, then the inductive Sylow-AW condition holds for $S$ and $\ell$.
\end{lem}

\begin{proof}
First note that $(\N_G(P)/P)/\Z(G)\cong \N_S(\overline P)/\overline P$.
If $\Aut(S)_{\overline P}$ acts trivially on $\N_S(\overline P)/\overline P$, then
$\N_S(\overline P)/\overline P$ is abelian.
By Lemma~\ref{lem:ele-lem}, $\N_G(P)/P\cong \Z(G)\times \N_S(\overline P)/\overline P$.
Thus, for every $\nu\in\Irr(\Z(G))$, the set $\IBr(\N_G(P)\mid\nu)$ has $|\N_S(\overline P)/\overline P|$ elements, all fixed by $\Aut(G)_{P,\nu}$.

Now assume that $\IBr(G\mid\nu)$ contains at least $|\N_S(\overline P)/\overline P|$ elements fixed by $\Aut(G)_\nu$ for every $\nu\in\Irr(\Z(G))$.
Then there exists an $\Aut(G)_{P,\nu}$-equivariant injection $\IBr(\N_G(P)\mid\nu)\embed\IBr(G\mid\nu)$.
Then the inductive Sylow-AW condition is implied by Lemma~\ref{lem:S3-cri}.	
\end{proof}	

\begin{cor}\label{cor:D4ex}
The inductive Sylow-AW condition holds for the simple group $\Omega^+_8(2)$.
\end{cor}	
	
\begin{proof}
Let $S=\Omega^+_8(2)$. Then the Schur multiplier of $S$ is ${C_2}^2$ and $\Out(S)\cong\fS_3$.

By\cite[Thm.~C]{NT11}, we assume that $\ell\ne 2$. 
Let $G$ be the $\ell'$-universal covering group of $S$, and $P$ a Sylow $\ell$-subgroup of $G$.
Then $|G|=2^{14}\cdot3^5\cdot5^2\cdot7$ and $\Z(G)\cong {C_2}^2$.
The action of $\Out(G)$ on $\Z(G)$ is the natural $\GL_2(2)$-action (\cite{CCNPW85}).
By \cite[Thm.~1.1]{KS16} we assume that $P$ is not cyclic.
Hence $\ell\in\{3,5\}$.

First we suppose that $\ell=3$.
Let $\overline P=P\Z(G)/\Z(G)\in\Syl_\ell(S)$.
The structure normalizer of Sylow 3-subgroups is known  (for example, see \cite[\S4]{Ma08}).
Precisely, direct calculation shows that $\N_S(\overline P)/\overline P\cong {C_2}^2$, and  $\Aut(S)_{\overline P}$ acts trivially on $\N_S(\overline P)/\overline P$ (this can also be checked by GAP \cite{GAP}).
We can check by \cite[p.233]{JLPW95} that for every $\nu\in\Irr(\Z(G))$, the set $\IBr(G\mid\nu)$ contains at least 4 Brauer charactes with distinct degrees.
Hence the inductive Sylow-AW condition holds for $S$ and 3 by Lemma~\ref{lem:S3-SylAW}.

Now let $\ell=5$.
Then $P$ is abelian.
According to Lemma~\ref{lem:S3-cri}, it suffices to show that for every $\nu\in\Irr(\Z(G))$, there exists an $\Aut(G)_{P,\nu}$-equivariant injection $\IBr(\N_{G}(P)\mid\nu)\embed\IBr(G\mid\nu)$.
Note that $\Out(G)$ acts transitively on $\Irr(\Z(G))\setminus\{1_{\Z(G)}\}$.
Using GAP \cite{GAP}, we can check that $\N_{G}(P)/P$ is of order 64.
Precisely, if $\nu\in\Irr(\Z(G))\setminus\{1_{\Z(G)}\}$, then the image of $\Aut(G)_{P,\nu}\to\Out(G)$ is of order 2, and
$\IBr(\N_{G}(P)\mid\nu)$ has six $\Aut(G)_{P,\nu}$-fixed points and two $\Aut(G)_{P,\nu}$-orbits of length 2.
And $\IBr(\N_{G}(P)/\Z(G))$ has four $\Aut(G)_P$-fixed points and two $\Aut(G)_P$-orbits of length 3.
Thus it suffices to show that $\IBr(G\mid \nu)$ and $\IBr(S)$ afford the same number of orbits of each length.
This can be obtained by \cite[p.233]{JLPW95}.
\end{proof}	

If a group $X$ acts on a set $\Omega$, then we denote by $\Omega^X$ the set of fixed points of $X$.

\begin{prop}\label{prop:2E6}	
The inductive Sylow-AW condition holds for the simple group ${}^2\ty E_6(2)$.
\end{prop}	
	
\begin{proof}
Let $S={}^2\ty E_6(2)$. Then $|S|=2^{36}\cdot3^9\cdot5^2\cdot7^2\cdot11 \cdot13 \cdot17\cdot 19$, the Schur multiplier of $S$ is ${C_2}^2\ti C_3$ and $\Out(S)\cong\fS_3$.
By\cite[Thm.~C]{NT11}, we assume that $\ell\ne 2$. 
If the Sylow $\ell$-subgroups of the universal covering group of $S$ are cyclic, then by \cite[Thm.~1.1]{KS16} the  inductive Sylow-AW condition holds.
Hence we assume that $\ell\in\{3,5,7\}$.
Let $G$ be the universal $\ell'$-covering group of $S$. 

We first determine global information as in the following Table~\ref{table}. 
\begin{table}[htbp]
	\centering
	\[\begin{array}{c|c|c|r|c|c|c}
		\ell&o(\nu)&\Out(G)_\nu&|\IBr(G\mid\nu)|&|\IBr(G\mid\nu)^{\Out(G)_\nu}|&b_2&b_3\\ \hline
		3&1&S_3&68&\geq16& & \\
		3&2&C_2&37&29&4&0\\ \hline
		5&1&S_3&115&\geq35&\geq6& \\
		5&2&C_2&65&53&6&0\\
		5&3&C_3&99&57&0&14\\
		5&6&1&49&49&0&0\\ \hline
		7&1&S_3&114&\geq38& & \\
		7&2&C_2&66&54&6&0\\
		7&3&C_3&99&60&0&13\\
		7&6&1&51&51&0&0
	\end{array}\]
	\caption{Action of $\Out(G)_\nu$ on $\IBr(G\mid\nu)$}  \label{table}
\end{table}
Here $b_r$ denotes the number of $\Out(G)_\nu$-orbits of length $r$ on $\IBr(G\mid\nu)$.
In Table~\ref{table2}, the number of $\ell$-regular classes are listed, by  $\mathbb{ATLAS}$ \cite[p.192--199]{CCNPW85} for the groups $S$, $S.2$, $2.S$ and $2.S.2$, and by the character table library of of GAP \cite{GAP} for the groups $3.S$ and $6.S$.
\begin{table}[htbp]
	\centering
\[\begin{array}{c|r|r|r|r|r|r}
	\ell &S&S.2&2.S&2.S.2&3.S&6.S\\ \hline
	3&68&97&105&159& & \\
	5&115&170&180&282&313&476\\
	7&114&171&180&285&312&480
\end{array}\]
	\caption{Number of $\ell$-regular classes}  \label{table2}
\end{table}
Let $d$ be a generator of $\textup O_3(\Out(G))\cong C_3$.
By GAP \cite{GAP}, the number of $d$-invariant $\ell$-conjugacy classes of $3.S$ is $187$ or $195$ according as $\ell=5$ or $7$.
By Brauer permutation lemma (\cite[Lemma 7.2]{GHKMW93}), $|\IBr(3.S)^{d}|=187$ or $195$ according as $\ell=5$ or $7$.

Let $\ell=3$.
Then $\Z(G)\cong {C_2}^2$ and the action of $\Out(G)$ on $\Z(G)$ is the natural $\GL_2(2)$-action.
First let $\nu=1_{\Z(G)}$.
Recall that we denote by $l(H)$ the number of conjugacy classes of $\ell$-regular elements of a finite group $H$.
Then $|\IBr(S)|=68$ and $|\IBr(S.2)|=97$.
It follows from Clifford theroy that $|\IBr(S)^{S.2}|=42$.
Since $\Out(G)$ is generated by two transpositions, we get that \[|\IBr(S)^{\Out(G)}|\ge 2|\IBr(S)^{S.2}|-|\IBr(S)|=16.\]
Now assume that $\nu\ne 1_{\Z(G)}$.
Let $Y=G/\ker(\nu)=2.S$.
Then $|\IBr(Y\mid\nu)|=37$ and $|\IBr(Y.2\mid\nu)|=62$.
So $|\IBr(Y\mid\nu)^{\Out(G)_\nu}|=29$.
Thus we get the two lines of Table~\ref{table} for the prime $\ell=3$.

The above arguments also apply for primes $\ell=5$ and $7$ for the case $o(\nu)=1$ or $2$, except when $\ell=5$ and $\nu=1_{\Z(G)}$ we need to show $b_2\ge 6$.
We mention that, under the ordinary character labels in the $\mathbb{ATLAS}$ \cite[p.192--196]{CCNPW85}, the 12 characters
$(\chi_{58},\chi_{59})$,
$(\chi_{95},\chi_{96})$,
$(\chi_{107},\chi_{108})$,
$(\chi_{109},\chi_{110})$,
$(\chi_{118},\chi_{119})$,
$(\chi_{123},\chi_{124})$ 
are all invariant under $\textup O_3(\Out(S))$, and 
the two characters in each pair are exchanged by $S.2$ (this can be also checked by GAP \cite{GAP}).
These 12 characters are of $5$-defect zero.
Hence $b_2\ge 6$.

Now assume that $\ell\in\{5,7\}$ and $o(\nu)\in\{3,6\}$.
If $o(\nu)=6$, then $\Out(G)_{\nu}=1$, and the numbers can be computed directly.
Now assume that $o(\nu)=3$.
By \cite[\S4.2]{Br16}, the number of the $\textup O_3(\Out(S))$-orbits of length 3 on the conjugacy classes of $S$ is $14$.
By Brauer permutation lemma, $|\IBr(S)^{\textup O_3(\Out(S))}|=73$ or $75$ according as  $\ell=5$ or $7$.
From this, we can obtain all the numbers in Table~\ref{table}.

Now we consider the local situation and establish the inductive Sylow-AW condition.
Let $P$ be a Sylow $\ell$-subgroup of $G$.
Write $\overline P=P\Z(G)/\Z(G)$.
Then \[|\IBr(\N_G(P)\mid \nu)|=|\Irr(\N_G(P)/P\mid \nu)|\le |\Irr(\N_S(\overline P)/\overline P)|,\]
where the inequality follows by a result of Gallagher \cite[Chap.~28, Thm.~1.3]{Ka92}.

First let $\ell=3$.
By \cite[p.~460]{Ma08}, the group $S$ has a subgroup $L\cong\Omega_7(3)$ such that $\N_S(\overline P)\le L$ and $\N_{S.2}(\overline P)\le L.2$.
Therefore, $\N_{S.2}(\overline P)/\overline P=\N_{L.2}(\overline P)/\overline P\cong {C_2}^3$ and $\N_{S}(\overline P)/\overline P\cong {C_2}^2$; in particular, $\N_{S.2}(\overline P)/\overline P$ acts trivially on $\N_S(\overline P)/\overline P$.
By the fact that $\Out(G)$ is generated by two transpositions again, $\Aut(S)_{\overline P}$ acts trivially on $\N_S(\overline P)/\overline P$.
Hence Lemma~\ref{lem:S3-SylAW} applies, and the inductive Sylow-AW condition holds for $S$ and the prime 3.

Now assume that $\ell=5$.
By \cite[\S4]{Ma08} and \cite[\S4]{Wi18}, $\C_S(\overline P)/\overline P\cong C_3$,
$\N_S(\overline P)/\overline P\cong (C_3\times 4.\fA_4)\rtimes C_2$, $\N_S(\overline P)/\C_S(\overline P)\cong G_8$, $\N_{S.2}(\overline P)/\overline P\cong S_3\ti G_8$, where $G_8$ is the reflection group in terms of Shephard–Todd labels \cite{ST54}.
Note that $\Irr(\N_S(\overline P)/\overline P)=30$ and recall that $\Irr(G_8)=16$, and consequently, $\Irr(\N_{S.2}(\overline P)/\overline P)=48$.
In particular, $|\IBr(\N_G(P)\mid \nu)|\le 30$ for every $\nu\in\Irr(\Z(G))$.

If $o(\nu)=6$, then $\Out(G)_\nu=1$.
Then $|\IBr(\N_G(P)\mid \nu)|\le|\IBr(G\mid \nu)|$ yields an  $\Out(G)_\nu$-equivariant injection from $\IBr(\N_G(P)\mid \nu)$ to $\IBr(G\mid \nu)$.
If $o(\nu)=3$, then $|\Out(G)_\nu|=3$.
The number of $\Out(G)_\nu$-fixed points on $\IBr(\N_G(P)\mid \nu)$ is bounded by 30, and the number of $\Out(G)_\nu$-orbits of length 3 on $\IBr(\N_G(P)\mid \nu)$ is bounded by 10.
Then by Table~\ref{table}, there exists an  $\Out(G)_\nu$-equivariant injection from $\IBr(\N_G(P)\mid \nu)$ to $\IBr(G\mid \nu)$.

Now assume that $o(\nu)=1$ or $2$.
Clifford theory implies that $|\Irr(\N_{S}(\overline P)/\overline P)^{S.2}|=22$.
By \cite[\S4]{Wi18}, $\Aut(S)_{\overline P}/\overline P\cong ((F\rtimes C_2)\rtimes 4.\fA_4)\rtimes C_2$, where $F=\textup{O}_3(\Aut(S)_{\overline P}/\overline P)\cong {C_3}^2$.
Here, $\C_S(\overline P)\unlhd \Aut(S)_{\overline P}$ and $F\cap (\N_S(\overline P)/\overline P)=\C_S(\overline P)/\overline P$.
We consider the action of $\Aut(S)_{\overline P}/\overline P$ on $F$, then $4.\fA_4$ is contained in the kernel.
Maschke's theorem implies that there exists a subgroup $K$ of $F$ such that $F=K\ti (\C_S(\overline P)/\overline P)$.
Then $K$ commutes with $\N_S(\overline P)/\overline P$, and the image of $K$ in $\Aut(S)_{\overline P}/\N_S(\overline P)\cong\fS_3$ has order 3.
Hence by Lemma~\ref{lem:ele-lem},
$\N_G(P)/P\textup{O}_3(\Z(G))=\textup{O}_2(\Z(G))\ti \C_{\N_G(P)/P\textup{O}_3(\Z(G))}(K)$.
Also there is an $\Aut(G)_{\overline P}/\overline P$-equivariant isomorphism between $\C_{\N_G(P)/P\textup{O}_3(\Z(G))}(K)$ and $\N_S(\overline P)/\overline P$.
Hence $\N_G(P)/P\textup{O}_3(\Z(G))\cong \textup{O}_2(\Z(G))\ti \N_S(\overline P)/\overline P$.
Therefore, there exists an $\Out(G)_\nu$-bijection $\IBr(\N_G(P)\mid\nu)\to \Irr(\N_S(\overline P)/\overline P)$.
As $K$ commutes with $\N_S(\overline P)/\overline P$, we can compute the orbit lengths of $\Out(G)_\nu$ on $\IBr(\N_G(P)\mid\nu)$ from the action of $\N_{S.2}(\overline P)/\overline P$ on $\N_S(\overline P)/\overline P$.
By direct calculation, $\IBr(\N_G(P)\mid\nu)$ has 22 fixed points and 4 orbits of length 2 under the action of $\Out(G)_\nu$.
From this, we get an $\Out(G)_\nu$-equivariant injection from $\IBr(\N_G(P)\mid \nu)$ to $\IBr(G\mid \nu)$, which implies the inductive Sylow-AW condition holds for $S$ and the prime 5.

Finally let $\ell=7$.
By \cite[\S4]{Wi18}, $\N_S(\overline P)/\overline P\cong\SL_2(3)\ti C_3$ and $\N_{\Aut(S)}(\overline P)\cong\fS_3\ti \N_S(\overline P)$.
In particular, $|\IBr(\N_G(P)\mid\nu)|\le 21$ for every $\nu\in\Irr(\Z(G))$.
If $o(\nu)=3$ or $6$, then similar to the above arguments, we can get an $\Out(G)_\nu$-equivariant injection from $\IBr(\N_G(P)\mid \nu)$ to $\IBr(G\mid \nu)$.
So we assume that $o(\nu)=1$ or $2$.
Note that $\N_S(\overline P)$ has no quotient group of order 2, hence its image in $\Out(S)\cong\fS_3$ is contained in $\fA_3$.
On the other hand, the image of $\N_{\Aut(S)}(\overline P)$ in $\Out(S)$ is the whole group, this implies that  the image of $\N_S(\overline P)$ in $\Out(S)$ is trivial.
Thus $\textup O_3(\Out(G))$ acts on $\N_G(P)/P\Z(G)$, and Lemma~\ref{lem:ele-lem} implies that $\N_G(P)/P\textup O_3(\Z(G))=\textup O_2(\Z(G))\ti \C_{\N_G(P)/P\textup O_3(\Z(G))}(\textup O_3(\Out(G)))$
and $\C_{\N_G(P)/P\textup O_3(\Z(G))}(\textup O_3(\Out(G)))$ can be identified with $\N_S(\overline P)/\overline P$ as above.
Moreover, $\textup O_3(\Out(G))$ acts trivially on $\N_S(\overline P)/\overline P$.
From this, we deduce that $\IBr(\N_G(P)\mid\nu)$ contains 21 Brauer characters,  all of which are fixed by $\Out(G)_\nu$.
Therefore, we get an $\Out(G)_\nu$-equivariant injection from $\IBr(\N_G(P)\mid \nu)$ to $\IBr(G\mid \nu)$, which implies the inductive Sylow-AW condition holds for $S$ and the prime 7.
\end{proof}	

\begin{thm}\label{thm:exc-covering}
Let $S$ be a simple group of Lie type that has an exceptional Schur multiplier.
Then the inductive Sylow-AW condition holds for $S$.
\end{thm}

\begin{proof}
For the list of simple groups of Lie type with an exceptional Schur multiplier, we refer to \cite[Table~6.1.3]{GLS98}.
By \cite{Sc16},	\cite{FLZ21} and \cite[\S8]{AHL24}, we only need to consider the groups 
$\Omega_7(3)$, $\Omega^+_8(2)$ and $^2\ty E_6(2)$.
For the universal covering group of $\Omega_7(3)$, its Sylow $\ell$-subgroups are cyclic unless $\ell=2,3$, and thus the inductive Sylow-AW condition follows by \cite[Thm.~C]{NT11} and \cite[Thm.~5]{MNT23}.
Therefore, Corollary~\ref{cor:D4ex} and Proposition~\ref{prop:2E6}	complete the proof.	
\end{proof}

Therefore, by \cite[Thm.~C]{NT11}, \cite[Thm.~1.1]{AD12}, \cite[Thm.~1.1]{Ma14}, \cite[Thm.~5]{MNT23} and Theorem~\ref{thm:exc-covering}, the inductive Sylow-AW condition only remains open for simple group of Lie type without exceptional Schur multiplier and odd non-defining characteristic.

%%%%%%%%%%%%%%%%%%%%%%%%%%%%%%%%%%%%%%%%%%%%%%%%%%%%%%%%%%%%%%%%%%%%%%%%
\section{Groups of Lie type}\label{sec:lie}
	
In this section, we first recall some results on the representation theory of finite reductive groups, and then construct an injection from the irreducible characters of the normalizer of a Sylow $e$-torus to those of the finite reductive group.
	
\subsection{An equivariant injection for finite reductive groups with connected center}	
	
Let $\bG$ be a connected reductive group over $\overline{\mathbb F}_p$ for a prime~$p$
	and let $F\colon\bG\to \bG$ be a Frobenius endomorphism defining an
	$\FF_q$-structure on $\bG$, where~$q$ is a power of~$p$.
Let $\bG^*$ be Langlands dual to $\bG$ with corresponding Frobenius endomorphism also denoted by $F$.

We refer to \cite{GM20} for notation pertaining to the character theory of finite reductive groups.
For a semisimple element $s$ of ${\bG^*}^F$, we denote by $\cE(\bG^F,s)$ the Lusztig series corresponding to $s$.
Then by \cite[7.6]{Lu77}, there is a partitioning
\[
\Irr(\bG^F)=\coprod_s\cE(\bG^F,s)
\]
where $s$ runs through the ${\bG^*}^F$-conjugacy classes of the semisimple elements of ${\bG^*}^F$.
Recall that the characters in $\cE(\bG^F,1)$ are called unipotent characters of $\bG^F$.

If $\Z(\bG)$ is connected, then by \cite[4.23]{Lu84}, there is a bijection, called Jordan decomposition, 
\[\sJ^{\bG}_s\colon\cE(\C_{\bG^*}(s)^F,1)\to\cE(\bG^F,s)\]
where $\C_{\bG^*}(s)$ is a connected reductive group.
Moreover, we choose the Jordan decomposition $\sJ^{\bG}_s$ as in \cite[Thm.~7.1]{DM90} (see also \cite[Thm.~4.7.1]{GM20}) so that certain properties are satisfied.

For any parabolic subgroup \(\bP\) of \(\bG\) admitting a Levi decomposition with some \(F\)-stable Levi subgroup \(\bL\), we have the Deligne--Lusztig induction functor \[\R_{\bL \leq \bP}^\bG : \mathbb{Z} \Irr(\bL^F) \to \mathbb{Z} \Irr(\bG^F)\] and its adjoint \({}^*\R_{\bL \leq \bP}^\bG\) for the scalar product of characters (see \cite[\S3.3]{GM20}). When applied to unipotent characters, those functors \(\R_{\bL \leq \bP}^\bG\) and \({}^*\R_{\bL \leq \bP}^\bG\) are independent of the parabolic subgroup $\bP$ by \cite[1.33]{BMM93}, so we write simply \(\R_\bL^\bG\) and \({}^*\R_\bL^\bG\).

For a positive integer $e$, we denote by $\phi_e$ the $e$-th cyclotomic polynomial. We will make use of the terminology of Sylow $e$-theory, introduced in \cite{BM92} (see also \cite[\S25]{MT11}). For $\bT$ an $F$-stable torus, $\bT_{\phi_e}$ denotes its Sylow $e$-torus.

Let $E\subseteq \ZZ_{\ge 1}$. Recall that an \emph{$E$-torus} of $\bG$ is an $F$-stable torus whose polynomial order is a product of cyclotomic polynomials in $\{\phi_e\mid e\in E\}$ and an \emph{$E$-split Levi subgroup} of $\bG$ is the centralizer of an $E$-torus of $\bG$.
We say that an irreducible character $\chi\in\Irr(\bG^F)$ is \emph{$E$-cuspidal} if $^*\R^\bG_{\bL\subseteq \bP}(\chi)=0$ for all proper $E$-split Levi subgroups $\bL$ of $\bG$ and any parabolic subgroup $\bP$ of
$\bG$ containing $\bL$ as a Levi complement. 
If $\bL\le\bG$ is $E$-split and $\lambda\in\Irr(\bL^F)$ is $E$-cuspidal, then $(\bL,\lambda)$ is called an \emph{$E$-cuspidal pair of $\bG$}.
If $E=\{ e\}$, then we also say \emph{$e$-split} and \emph{$e$-cuspidal} for $E$-split and $E$-cuspidal respectively.

Let $e$ be a positive integer and $(\bL,\lambda)$ be a unipotent $e$-cuspidal pair of $\bG$. 
Denote by $\cE(\bG^F,(\bL,\lambda))$ the corresponding $e$-Harish-Chandra series, that is the set of irreducible components of $\R_{\bL}^{\bG}(\lambda)$.
Write $W_{\bG^F}(\bL,\lambda)$ for the relative Weyl group, that is, the stabiliser of $\lambda$ in $\N_{\bG^F}(\bL)/\bL^F$.
By a result of Brou\'e--Malle--Michel \cite[Thm.~3.2]{BMM93} there is a bijection
\[\fI_{\bL,\lambda}^{\bG}\colon \Irr(W_{\bG^F}(\bL,\lambda))\to \cE(\bG^F,(\bL,\lambda))\]
such that the Comparison Theorem holds.

Denote by \( \cE(\bG^F, \ell') \) the set of irreducible characters of \( \bG^F \) lying in a Lusztig series \( \cE(\bG^F, s) \), where \( s \in \bG^{*F} \) is a semisimple \( \ell' \)-element. 
By \cite{Ge93,GH91}, \( \cE(\bG^F, \ell') \) forms a basic set, provided that \( \ell \) is good for \( \bG \) and does not divide the defining characteristic of \( \bG \), and that \( \ell \) does not divide the order of \( (\Z(\bG)/\Z^\circ(\bG))_F \) (the largest quotient on which \( F \) acts trivially).

\begin{prop}\label{prop:Syl}
	Let \(\bH\) be simple, simply connected, defined over \(\mathbb{F}_q\) with corresponding Frobenius map \(F : \mathbf{H} \to \mathbf{H}\). Let \(\ell \) be an odd prime dividing \(|\bH^F|\), good for $\bH$, not dividing \(q|\Z(\bH)|\). 
	Let $e=e_\ell(q)$.
	Assume further that when $\ell = 3$, the group $\mathbf{H}^F$ is not one of the following:
	\begin{itemize}
		\item $\bH^F$ is of type $^3\ty D_4$,
		\item $\mathbf{H}^F = \mathrm{SL}_3(q)$ with $q \equiv 4, 7 \pmod{9} $,
		\item $\mathbf{H}^F = \mathrm{SU}_3(q)$ with $q \equiv 2, 5 \pmod{9} $,
		\item $ \mathbf{H}^F = \ty G_2(q) $ with $ q \equiv 2, 4, 5, 7 \pmod{9}$.
	\end{itemize}
	Let $P$ be a Sylow $\ell$-subgroup of $\bH^F$. 
	Then 
	\begin{enumerate}[\rm(i)]
		\item The group $P$ is Cabanes. 
		\item There is a unique Sylow $e$-torus $\bS$ of $\bH$ containing $\Cab(P)$. In particular, $\N_{\bH^F}(P)\subseteq \N_{\bH}(\bS)$.
		\item Let $\bL=\C_{\bH}(\bS)$. Then $\Cab(P)=\bS^F_\ell=\Z(\bL)^F_\ell=\bL^F\cap P$ is a normal Sylow $\ell$-subgroup of~$\bL^F$. In particular, $\bL^F=\Cab(P)\times\textup{O}_{\ell'}(\bL^F)$.
		\item One has $\N_{\bL^F}(P)=\Cab(P)\C_{\bL^F}(P)$.
		\item If $\ell$ is good for $\bH$ and does not divide $|\Z(\bH)^F|$, then $\bL^F=\C_{\bH^F}(\Cab(P))$ and $\N_{\bL^F}(P)=\Cab(P)\C_{\bH^F}(P)$.
		\item 
		One has	$\Irr(\bL^F\mid 1_{\Z(\bL)^F_\ell})=\cE(\bL^F,\ell')$. 
	\end{enumerate}
\end{prop}

\begin{proof}
	The assertions (i) and (ii) are proved in the proof of \cite[Thm.~5.14]{Ma07}, based on \cite[Thm.~4.4]{Ca94} and \cite[Prop.~5.13]{Ma07}. 
	Now we consider (iii). Write $A=\Cab(P)$.
	It follows from $P\subseteq \N_{\bH}(\bS)$ that $\bS^F\cap P$ is a Sylow $\ell$-subgroup of $\bS^F$ and an abelian normal subgroup of $P$.
	Hence $A=\bS^F_\ell$.
	Since $\bL^F\cap P$ is a Sylow $\ell$-subgroup of $\bL^F$, we have $\Z(\bL)^F_\ell\subseteq \bL^F\cap P\subseteq \C_P(A)$. 
	By Lemma~\ref{lem:self-cen}, $\C_P(A)=A$.
	Then it follows from $\bS\subseteq \Z(\bL)$ that 
	$A=\bS^F_\ell=\Z(\bL)^F_\ell=\bL^F\cap P$.
	By the Schur--Zassenhaus Theorem, $A$ has a complement in $\bL^F$. 
	As $A$ is a central subgroup of $\bL^F$, we have $\bL^F=A\times\textup{O}_{\ell'}(\bL^F)$.
	This gives (iii).
	
	For (iv), we note that $\bL^FP=P\textup{O}_{\ell'}(\bL^F)$, so
	$\N_{\bL^FP}(P)=P\N_{\textup{O}_{\ell'}(\bL^F)}(P)=P\times\N_{\textup{O}_{\ell'}(\bL^F)}(P)$.
	Therefore, $\N_{\textup{O}_{\ell'}(\bL^F)}(P)=\C_{\textup{O}_{\ell'}(\bL^F)}(P)$ and $\N_{\bH^F}(P)\cap\bL^F=A\times\C_{\textup{O}_{\ell'}(\bL^F)}(P)=A\C_{\bL^F}(P)$.
	The statement (v) follows from \cite[Prop.~2.2]{CE94},
	and (vi) follows by construction.
\end{proof}
	
\begin{cor}\label{cor:syl-Levi-syl}
	Let \(\bH\) be simple, simply connected, defined over \(\mathbb{F}_q\) with corresponding Frobenius map \(F : \mathbf{H} \to \mathbf{H}\). Let \(\ell \) be an odd prime dividing \(|\bH^F|\), good for $\bH$, not dividing \(q|\Z(\bH)|\). 
Let $e=e_\ell(q)$.
Let $\bL$ be the centralizer of some Sylow $e$-torus of $\bH$.	
	Then $\Z(\bL)^F_\ell$ is a normal Sylow $\ell$-subgroup of $\bL^F$.
\end{cor}

\begin{proof}
	If $\ell=2$ or $3$, then $e\in\{1,2\}$, and thus $\bL$ is a torus by \cite[Lemma~3.17]{KM15}.
	Hence we assume $\ell\ge 5$.
	Then this assertion follows from Proposition~\ref{prop:Syl}.
\end{proof}

We propose the follow question:

\begin{question}
	Does Corollary~\ref{cor:syl-Levi-syl} hold for an arbitrary connected reductive group $\bH$?	
\end{question}

\subsection{An equivariant injection}\label{subsec:irr-nor-tori}
	
	Now let $\bG$ be a simple algebraic group of simply connected type over $\overline\FF_p$. 
	Let $\Phi$ and $\Delta$ denote respectively the set of
	roots and simple roots of $\bG$ determined by the choice of a maximal torus and
	a Borel subgroup containing it. To describe Frobenius endomorphisms of~$\bG$,
	we use the Chevalley generators $x_\alpha(t)$
	($t\in\overline\FF_q$, $\alpha\in\Phi$) as in~\cite[Thm.~1.12.1]{GLS98}.
	
	Recall the endomorphisms of $\bG$ described as in~\cite[\S2]{MS16}.
	Let $F_0:\bG\to\bG$ denote the field endomorphism of $\bG$ given by
	$F_0(x_\alpha(t))=x_\alpha(t^p)$ for $t\in\overline\FF_p$ and $\alpha\in\Phi$.
	A (length-preserving) automorphism $\tau$ of the Dynkin diagram associated to
	$\Delta$ (and hence an automorphism of $\Phi$) determines a graph automorphism
	$\gamma$ of $\bG$ given by $\gamma(x_\alpha(t)):=x_{\tau(\alpha)}(t)$ for
	$t\in\overline\FF_p$ and $\alpha\in \pm\Delta$. Any such $\gamma$ commutes with $F_0$.
	
	Suppose that $\Z(\bG)$ has rank $r$ as a finite abelian group.
	Let $\mathbb Z$ be a torus of rank~$r$ with an embedding of
	$\Z(\bG)$. Let us set $\tbG:=\bG\times_{\Z(\bG)}\mathbb Z$ the central product of
	$\bG$ and $\mathbb Z$ over $\Z(\bG)$. Then $\tbG$ is a connected reductive group such
	that the natural map $\bG\hookrightarrow\tbG$ is a regular embedding (cf. \cite[\S7]{Lu88}).
	As in \cite[p.~874]{MS16}, we can extend $F_0$ to a Frobenius endomorphism of
	$\tbG$ and $\gamma$ to an automorphism of $\tbG$.
	
	Consider a Frobenius endomorphism $F:=F_0^f\gamma$, with $f$ a positive integer
	and $\gamma$ a (possibly trivial) graph automorphism of $\bG$, leaving aside the types $^2{\ty{B}}_2$, $^2{\ty{G}}_2$ and $^2{\ty{F}}_4$.
	Then $F$ defines an $\FF_q$-structure on $\tbG$, where $q=p^f$.
	The groups of rational points $G=\bG^F$ and $\tG=\widetilde\bG^F$ are
	finite. 
	Let $\cB$ be the subgroup of $\Aut(\bG^F)$ generated by $F_0$ (here we identify
	$F_0$ with $F_0|_G$) and the graph automorphisms commuting with $F$.
	Then $\tbG^F\rtimes \cB$ is well defined and induces all automorphisms of
	$\bG^F$ (see~\cite[Thm.~2.5.1]{GLS98}). 
	
Let $e=e_\ell(q)$ and $\bS$ be a Sylow $\phi_e$-torus of $\bG$. Let $\tbS=\Z(\tbG)_{\phi_e}\bS$. Then  $\tbS$ is a Sylow $e$-torus of $\tbG$ since $\tbG=\Z(\tbG)\bG$.
Write $\bL=\C_{\bG}(\bS)$ and $\tbL=\C_{\tbG}(\bS)$.
Then $\tbL=\C_{\tbG}(\tbS)=\Z(\tbG)\bL$ and $\bL=\tbL\cap \bG$.
Let $\bS^*$ (reps. $\tbS^*$) be a Sylow $e$-torus of $\bG^*$ (resp. $\tbG^*$) in duality with $\bS$ (resp. $\tbS$).

\begin{prop}\label{prop:max-ext}
	\begin{enumerate}[\rm(i)]
\item 	There is an $\N_{\bG^F\cB}(\bS)$-equivariant extension map for  \(\Irr(\bL^F)\) with respect to \( \bL^F\unlhd \N_{\bG^F}(\bS) \).
\item 	There is an $\N_{\tbG^F \cB}(\tbL)$-equivariant extension map $\widetilde\Lambda$ for $\Irr(\tbL^F)$ with respect to $\tbL^F\unlhd \N_{\tbG^F}(\bS)$ such that for every \(\xi \in \Irr(\tbL^F)\) and \(\delta \in \Irr(\tbG^F/\bG^F)\),
\[\widetilde\Lambda(\xi \operatorname{Res}_{\tbL^F}^{\tbG^F}(\delta)) = \widetilde\Lambda(\xi) \operatorname{Res}_{\N_{\tbG^F}(\bS)_{\xi}}^{\tbG^F}(\delta).\] 
\end{enumerate}
\end{prop}

\begin{proof}
First we consider (i).	
Maximal extendibility holds for \(\Irr(\bL^F)\) with respect to \( \bL^F\unlhd \N_{\bG^F}(\bS) \).
This is established in \cite[Thm.~A]{Sp09} for exceptional groups, in \cite[Thm.~1.1]{Sp10a} for classical groups when $\bL^F$ is abelian, and in \cite[Thm.~1.1]{Sp10b} for the primes $\ell>3$.
Note that if $\ell\in\{ 2, 3 \}$, then $e\in\{ 1, 2 \}$ and thus $\bL$ is a torus by \cite[Lemma~3.17]{KM15}.	
Moreover, by \cite[Prop.~5.9]{CS17a}, \cite[Thm.~4.3]{CS17b}, \cite[\S4--\S6]{CS19}
and \cite[\S6]{CS26}, there exists a $\N_{\bG^F\cB}(\bS)$-equivariant extension map for $\Irr(\bL^F)$ with respect to $\bL^F\unlhd \N_{\bG^F}(\bS)$; recall that for the case that $\Z(\bG)$ is connected, this follows from \cite[\S5 and \S6]{CS13}.

The statement (ii) is the $B(d)$ condition defined in \cite[Definition~2.2]{CS19}. By (i) we may assume that $\Z(\bG)$ is disconnected.
Then the $B(d)$ condition follows from
\cite[Cor.~5.14]{CS17a}, \cite[Thm.~6.1]{CS17b}, \cite[\S4-\S6]{CS19} and \cite[\S6]{CS26}.
\end{proof}

	In this section, we will prove the following theorem.

\begin{thm}\label{thm:equiv}
\begin{enumerate}[\rm(1).]
		\item There exists an $\N_{\tbG^F \cB}(\bS)$-equivariant injection
	\[  \widetilde\Psi\colon\Irr(\N_{\tbG^F}(\bS))\embed\Irr(\tbG^F) \]
	such that
	\begin{enumerate}[\rm(i)]	
		\item for every semisimple element
		$\ts\in\C_{{\tbG}^{*F}}({\tbS}^*)$, \[\widetilde\Psi(\Irr(\N_{\tbG^F}(\bS)\mid\cE(\C_{{\tbG}^F}(\bS),\ts)))\subseteq\cE(\tbG^F,\ts),\] 
		\item	$\widetilde\Psi(\Irr(\N_{\tbG^F}(\bS)\mid\widetilde\lambda))\subseteq \Irr(\tbG^F\mid\widetilde\lambda)$ for every $\widetilde\lambda\in\Irr(\Z(\tbG)^F)$, and
		\item $\widetilde\Psi(\widetilde\eta\Res^{\tbG^F}_{\N_{\tbG^F}(\bS)}(\delta))=\widetilde\Psi(\widetilde\eta)\delta$ for every $\widetilde\eta\in\Irr(\N_{\tbG^F}(\bS))$ and $\delta\in\Irr(\tbG^F/\bG^F)$.
	\end{enumerate} 	
	\item \begin{enumerate}[\rm(i)]
		\item Maximal extendibility holds for \(\Irr(\bG^F)\) with respect to \( \bG^F\unlhd \tbG^F \).
		\item Maximal extendibility holds for $\Irr(\N_{\bG^F}(\bS))$ with respect to $\N_{\bG^F}(\bS)\unlhd \N_{\tbG^F}(\bS)$.
	\end{enumerate}
	\item \begin{enumerate}[\rm(i)]
		\item  For every \( \widetilde\chi \in \Irr(\tbG^F) \), there exists some \( \chi_0 \in \Irr(\bG^F \mid \widetilde\chi) \) such that
		\begin{enumerate}[\rm(a)]
			\item \((\tbG^F \cB)_{\chi_0} = \tbG^F_{\chi_0}  \cB_{\chi_0}\) and
			\item \( \chi_0 \) extends to its stabiliser \(\bG^F  \cB_{\chi_0}\).
		\end{enumerate}
		\item 
		For every \( \widetilde\eta \in \Irr(\N_{\tbG^F}(\bS)) \), there exists some \( \eta_0 \in \Irr(\N_{\bG^F}(\bS) \mid \widetilde\eta) \) such that
		$O=\bG^F\N_{\tbG^F\cB}(\bS)_{\eta_0}$
		satisfies
		\begin{enumerate}[\rm(a)]
			\item 
			$O=(\tbG^F\cap O)(\cB\cap O)$, and 
			\item
			$\eta_0$ extends to its stabiliser	$\N_{\bG^F\cB}(\bS)_{\eta_0}$.
		\end{enumerate}
	\end{enumerate}
	\end{enumerate}
\end{thm}
	
\begin{proof}
We first prove (2). 	Note that we can assume that $(\bG,F)$ is of type $\ty D_n$ with even $n\ge 4$, since otherwise $\tbG^F/\Z(\tbG^F)\bG^F$ is cyclic.	
Thus the  statement (2.i) follows by \cite[\S16]{CE04} and \cite[(f)]{Lu08}, while (2.ii) follows by \cite[\S and \S6]{CS26}.
	
The statement (3) is the $A(\infty)$ and $A(d)$ conditions defined in \cite[Definition~2.2]{CS19}.
The  statement (3.i) is proved in \cite[Thm.~4.1]{CS17a}, \cite[Thm.~3.1]{CS17b}, \cite[Thm.~B]{CS19} and \cite[Thm.~A]{Sp25}, while (3.ii) is proved in \cite[Thm.~5.1]{CS17a},  \cite[Thm.~3.1]{CS17b}, \cite[\S5 and \S6]{CS19} and \cite[Thm.~6.9]{CS26}.	
	
Now we consider (1).	
By \cite[Prop.~4.14(a)]{CS26}, one has an $\Irr(\tbG^F/\bG^F)\rtimes \N_{\bG^F\cB}(\bS)$ injection 
$\widetilde\Psi:\Irr(\N_{\tbG^F}(\bS))\to\Irr(\tbG^F)$ such that (1.ii) and (1.iii) hold.
Finally, we verify (1.i).
In fact, the construction of $\widetilde\Psi$ follows from the strategy of \cite[\S7]{Ma07}, \cite[Cor.~3.3]{CS13} and \cite[\S6]{CS19} and we recall it as follows.
Let $\widetilde\eta\in\Irr(\N_{\bG^F}(\bS))$ and $\widetilde\vartheta\in\Irr(\C_{{\tbG}^F}(\tbS)\mid\widetilde\eta)$.
Write $\widetilde\vartheta=\sJ_{\ts}^{\tbL}(\la)$ where $\ts$ is a semisimple element of $\tbL^{*F}$ and $\la$ is a unipotent character of $\C_{\tbG^{*F}}(\tilde s,\,\tbS^*)$.
On the other hand, if we let $\widetilde\chi=\widetilde\Psi(\widetilde\eta)$, then $\widetilde\chi=\sJ^{\tbG}_{\tilde s}(\fI^{\C_{\tbG^{*}}(\tilde s)}_{\C_{\tbG^{*}}(\tilde s,\,\tbS^*),\,\la}(\eta))$ where $\eta\in\Irr(W_{\C_{\tbG^{*F}}(\tilde s)}(\C_{\tbG^{*F}}(\tilde s,\,\tbS^*),\lambda))$.	Hence $\widetilde\chi\in \cE(\tbG^F,\ts)$.
\end{proof}

Now we establish the following result.
	
\begin{cor}\label{cor:syl-torus}
There exists an $\N_{\tbG^F \cB}(\bS)$-equivariant injection
\[  \Psi\colon\Irr(\N_{\bG^F}(\bS))\embed\Irr(\bG^F) \]
such that 
\begin{enumerate}[\rm(i)]
\item for every semisimple element $s\in\C_{{\bG^*}^F}(\bS^*)$, \[\Psi(\Irr(\N_{\bG^F}(\bS)\mid\cE(\bL^F,s)))\subseteq\cE(\bG^F,s),\] 
%\item  \( \Psi(\operatorname{Irr}(\N_{\bG^F}(\bS) \mid \lambda)) \subseteq \operatorname{Irr}(\bG^F \mid \lambda) \) for every \( \lambda \in \operatorname{Irr}(\Z(\bG^F)) \),  and
\item  for any $\eta\in\Irr(\N_{\bG^F}(\bS))$ and $\chi=\Psi(\eta)$,
\[((\tbG^F \cB)_\chi,\bG^F,\chi)
\succeq_{c}
(\N_{\tbG^F \cB}(\bS)_\eta,\N_{\bG^F}(\bS),\eta).\]		
\end{enumerate}
\end{cor}

\begin{proof}
This follows from Theorem~\ref{thm:equiv}.
In fact, the arguments are just the same as in the proof of \cite[Thm.~2.12]{Sp12}.		
\end{proof}

\begin{cond}\label{cond:group}
	Let $\bG$, $\tbG$, $\cB$, and $q$ be defined as in \S\ref{subsec:irr-nor-tori}.
	Let $\ell$ be an odd prime with $\ell\nmid q|\Z(\bG)|$ and good for $\bH$. 
	Assume further that when $\ell = 3$, the group $\mathbf{G}^F$ is not one of the following:
	\begin{itemize}
			\item $\bH^F$ is of type $^3\ty D_4$,
		\item \( \mathbf{G}^F = \mathrm{SL}_3(q) \) with \( q \equiv 4, 7 \pmod{9} \),
		\item \( \mathbf{G}^F = \mathrm{SU}_3(q) \) with \( q \equiv 2, 5 \pmod{9} \),
		\item \( \mathbf{G}^F = \ty G_2(q) \) with \( q \equiv 2, 4, 5, 7 \pmod{9} \).
	\end{itemize}
	Suppose that $P$ is a Sylow $\ell$-subgroup of $\bG^F$.
	Set $e=e_\ell(q)$.
	Let $\bS$ be the unique Sylow $e$-torus of $\bG$ containing $\Cab(P)$.
\end{cond}

\begin{prop}\label{prop:corr-nor}
Keep Condition~\ref{cond:group}.
Then there exists an $\N_{\tbG^F\cB}(P)$-equivariant injection
\[\Delta\colon\Irr(\N_{\bG^F}(P)\mid 1_P)\embed\Irr(\N_{\bG^F}(\bS)\mid \cE(\bL^F,\ell'))  \]
such that 
\begin{equation}\label{equ-nor}
	(\N_{\tbG^F\cB}(\bS)_\varphi,\N_{\bG^F}(\bS),\varphi)\succeq_{c} (\N_{\tbG^F\cB}(P)_\theta,\N_{\bG^F}(P),\theta)
	\addtocounter{thm}{1}\tag{\thethm}
\end{equation}
for every $\theta\in\Irr(\N_{\bG^F}(P)\mid 1_P)$ and $\varphi=\Delta(\theta)$.
\end{prop}

\begin{proof}
Let $\bL=\C_{\bG}(\bS)$.	Write $A=\Z(\bL)^F_\ell$.
By Proposition~\ref{prop:Syl}, one has that $A\unlhd \N_{\tbG^F\cB}(\bS)$.
So $\N_{\N_{\bG^F}(\bS)/A}(P/A)=\N_{\bG^F}(P)/A$
and $\N_{\N_{\tbG^F\cB}(\bS)/A}(P/A)=\N_{\tbG^F\cB}(P)/A$.
According to Theorem~\ref{thm-MC}, there exists an $\N_{\tbG^F\cB}(P)/A$-equivariant bijection
\[\mathtt f:\Irr_{\ell'}(\N_{\bG^F}(P)/A)\to \Irr_{\ell'}(\N_{\bG^F}(\bS)/A)\]
such that 
\begin{equation}\label{equ:nor2}
((\N_{\tbG^F\cB}(\bS)/A)_{\overline\varphi},\N_{\bG^F}(\bS)/A,\overline\varphi)\succeq_{c} 
((\N_{\tbG^F\cB}(P)/A)_{\overline\theta},\N_{\bG^F}(P)/A,\overline\theta)
\addtocounter{thm}{1}\tag{\thethm}
\end{equation}
for every $\overline\theta\in\Irr_{\ell'}(\N_{\bG^F}(P)/A)$ and $\overline\varphi=\mathtt f(\overline\theta)$.

Note that $A$ lies in the kernel of every character in $\Irr(\N_{\bG^F}(P)\mid 1_P)$.
Hence $\mathtt f$ induces an injection 
\[\Delta\colon\Irr(\N_{\bG^F}(P)\mid 1_P)\embed\Irr(\N_{\bG^F}(\bS))\]
which is $\N_{\tbG^F\cB}(P)$-equivariant.
Let $\theta\in\Irr(\N_{\bG^F}(P)\mid 1_P)$ and $\varphi=\Delta(\theta)$.
Let $\overline\theta$ and $\overline\varphi$ be the deflations of $\theta$ and $\varphi$ to $\N_{\bG^F}(P)/A$ and $\N_{\bG^F}(\bS)/A$ respectively.
Then (\ref{equ:nor2}) implies (\ref{equ-nor}) by the proof of \cite[Lemma~3.12]{NS14}.
Therefore, it remains to show that $\varphi\in\Irr(\N_{\bG^F}(\bS)\mid \cE(\bL^F,\ell')) $.
This follows from Proposition~\ref{prop:Syl}.
\end{proof}

\begin{cor}\label{cor:ord-central}
Keep Condition~\ref{cond:group}.
There exists an $\N_{\tbG^F \cB}(P)$-equivariant injection
\[  \Omega\colon\Irr(\N_{\bG^F}(P)\mid 1_P)\hookrightarrow\cE(\bG^F,\ell') \]
such that  for any $\theta\in\Irr(\N_{\bG^F}(P)\mid 1_P)$ and $\chi=\Omega(\theta)$,
	\[((\tbG^F \cB)_\chi,\bG^F,\chi)
	\succeq_{c}
	(\N_{\tbG^F \cB}(P)_\theta,\N_{\bG^F}(P),\theta).\]		
\end{cor}

\begin{proof}
This follows directly from Corollary~\ref{cor:syl-torus} and Proposition~\ref{prop:corr-nor}.	
\end{proof}

We recall the following result of \cite{FW26}.
\begin{thm}\label{prop:bas-set}
Let $\bG$, $\tbG$, $\cB$, and $q$ be defined as in \S\ref{subsec:irr-nor-tori}. Let $\ell$ be a prime good for $\bG$ and satisfying $\ell\nmid q$. 
\begin{enumerate}[\rm(i)]
	\item There exists a blockwise $(\textup{Lin}_{\ell'}(\tbG^F/\bG^F)\rtimes \cB)$-equivariant bijection between $\cE(\tbG^F,\ell')$ and $\IBr(\tbG^F)$.
    \item Assume further that $\ell\nmid |\Z(\bG)^F|$, then there exists a blockwise $(\tbG^F\rtimes\cB)$-equivariant bijection between $\cE(\bG^F,\ell')$ and $\IBr(\bG^F)$.    
\end{enumerate} 
\end{thm}

\begin{proof}
This is \cite[Thm.~3.5]{FW26}.
\end{proof}	

\begin{prop}\label{prop:equiv}
Keep Condition~\ref{cond:group}, and assume further that $\ell$ is good for $\bG$ and does not divide $|\Z(\bG)^F|$.	
\begin{enumerate}[\rm(1).]
\item There exists an $\N_{\tbG^F \cB}(P)$-equivariant injection
\[  \widetilde\Omega\colon\IBr(\N_{\tbG^F}(P))\hookrightarrow\IBr(\tbG^F) \]
such that
\begin{enumerate}[\rm(i)]	
\item	$\widetilde\Omega(\IBr(\N_{\tbG^F}(P)\mid\widetilde\lambda))\subseteq \IBr(\tbG^F\mid\widetilde\lambda)$ for every $\widetilde\lambda\in\IBr(\Z(\tbG)^F)$, and
\item $\widetilde\Omega(\widetilde\eta\Res^{\tbG^F}_{\N_{\tbG^F}(P)}(\delta))=\widetilde\Omega(\widetilde\eta)\delta$ for every $\widetilde\eta\in\IBr(\N_{\tbG^F}(P))$ and $\delta\in\IBr(\tbG^F/\bG^F)$.
\end{enumerate} 	
\item 
\begin{enumerate}[\rm(i)]
\item Maximal extendibility holds for \(\IBr(\bG^F)\) with respect to \( \bG^F\unlhd \tbG^F \).
\item Maximal extendibility holds for $\IBr(\N_{\bG^F}(P))$ with respect to $\N_{\bG^F}(P)\unlhd \N_{\tbG^F}(P)$.
\end{enumerate}
\item 
\begin{enumerate}[\rm(i)]
\item  For every \( \widetilde\chi \in \IBr(\tbG^F) \), there exists some \( \chi_0 \in \operatorname{IBr}(\bG^F \mid \widetilde\chi) \) such that
\begin{enumerate}[\rm(a)]
\item \((\tbG^F  \cB)_{\chi_0} = \tbG^F_{\chi_0}  \cB_{\chi_0}\) and
\item \( \chi_0 \) extends to its stabiliser \(\bG^F  \cB_{\chi_0}\).
\end{enumerate}
\item For every \( \widetilde\eta \in \IBr(\N_{\tbG^F}(P)) \), there exists some \( \eta_0 \in \operatorname{IBr}(\N_{\bG^F}(P) \mid \widetilde\eta) \) such that
$O=\bG^F\N_{\tbG^F\cB}(P)_{\eta_0}$
satisfies
\begin{enumerate}[\rm(a)]
\item $O=(\tbG^F\cap O)(\cB\cap O)$, and 
\item $\eta_0$ extends to its stabiliser $\N_{\bG^F\cB}(P)_{\eta_0}$.
\end{enumerate}
\end{enumerate}
\end{enumerate}
\end{prop}

\begin{proof}
According to Proposition~\ref{prop:Cli} and Corollary~\ref{cor:ord-central}, there exists an $\N_{\tbG^F \cB}(P)$-equivariant injection
\[\widetilde\Omega'\colon\Irr(\N_{\tbG^F}(P)\mid 1_{P \textup{O}_\ell(\Z(\tbG^F))})\hookrightarrow\cE(\tbG^F,\ell') \]
such that for any $\widetilde\theta\in\Irr(\N_{\tbG^F}(P)\mid 1_{P \textup{O}_\ell(\Z(\tbG^F))})$ and $\widetilde\chi=\widetilde\Omega'(\widetilde\theta)$,
\[((\tbG^F \cB)_{\widetilde\chi},\tbG^F,\widetilde\chi)
\succeq_{c}
(\N_{\tbG^F \cB}(P)_{\widetilde\theta},\N_{\tbG^F}(P),\widetilde\theta),\]		
and $\widetilde\Omega'(\widetilde\theta\Res^{\tbG^F}_{\N_{\tbG^F}(P)}(\la))=\widetilde\Omega'(\widetilde\theta)\la$ for every $\la\in\Irr(\tbG^F/\bG^F)$ of $\ell'$-order.
Therefore, by Theorem~\ref{prop:bas-set}, the statement (1) holds.

The statement (2.i) follows by \cite{Ge93}, while the statement (3.i) holds by \cite[Thm.~1.3]{FW26}.
The statements (2.ii) and (3.ii) follows from Proposition~\ref{prop:corr-nor} and (2.ii) and (3.ii) of Theorem~\ref{thm:equiv}.
Thus we complete the proof.
\end{proof}

Then we prove:

\begin{thm}\label{thm:lie-good}
Let $\bG$ be a simple algebraic group of simply connected type over $\overline\FF_p$, and $F:\bG\to\bG$ is a Steinberg endomorphism.	
Assume that $S:=\bG^F/\Z(\bG^F)$ is simple, and $\ell$ is a prime different from $p$ and good for $\bG$ satisfying that $\ell\mid |\bG^F|$.
Then the inductive Sylow-AW condition holds for $S$ and $\ell$.
\end{thm}
	
\begin{proof}
Recall from \cite[p.~1209]{MNT23} that the inductive Sylow-AW condition is implied by the inductive Alperin weight condition from \cite{NT11}.
Therefore, by \cite[Thm.~1.1]{Ma14}, \cite[Thm.~1]{FLZ21} and \cite[Thm.~A]{Sc16}, we may assume that $(\bG,F)$ is not of type $\ty A_n$, $^2\ty A_n$, $\ty G_2$ or $^3\ty D_4$, and $F$ is a Frobenius endomorphism.
In particular, $\bG^F$ is neither a Suzuki group nor a Ree group by \cite[Thm.~1.1]{Ma14}.
Then $\ell$ is odd, and does not divide $|\Z(\bG)^F|$.
Let $P$ be a Sylow $\ell$-subgroup of $\bG^F$.
Then the statements of Proposition~\ref{prop:equiv} hold.

Let $\bG\le\tbG$ be a regular embedding.
Note that $\ell$ does not divide $|\tbG^F/\Z(\tbG)^F\bG^F|$.
Thus we can establish the inductive Sylow-AW condition for $S$ and $\ell$, using a similar argument of the proof of \cite[Prop.~4.3]{FS23} or \cite[Thm.~3.3]{BS22}.
\end{proof}

\begin{cor}\label{cor:great7}
	Let $S$ be a nonabelian simple group and $\ell$ be a prime with $\ell\notin\{3,5\}$ and $\ell\mid |S|$. Then the inductive Sylow-AW condition holds for $S$ and the prime $\ell$.
\end{cor}

\begin{proof}
Thanks to \cite[Thm.~1.1]{AD12}, \cite[Thm.~1.1]{Ma14} and \cite[Thm.~5]{MNT23}, we may assume that $S$ is a simple group of Lie type and $\ell\ge 7$.  
Furthermore, by \cite[Thm.~C]{NT11} and Theorem~\ref{thm:exc-covering}, we may also assume that $S$ does not have an exceptional Schur multiplier and $\ell$ is a non‑defining characteristic, whence the assertion follows from Theorem~\ref{thm:lie-good}.  
\end{proof}

By \cite[Thm.~2]{MNT23}, Theorem B follow by Corollary~\ref{cor:great7} immediately.

\begin{cor}\label{cor:classical}
Let $S$ be a simple classical group. Then the inductive Sylow-AW condition holds for $S$.
\end{cor}

\begin{proof}
According to \cite[Thm.~5]{MNT23}, we may assume that $\ell$ is odd.
Thus, similar to the proof of Corollary~\ref{cor:great7}, the assertion follows from  Theorem~\ref{thm:exc-covering} and Theorem~\ref{thm:lie-good}.  
\end{proof}

Finally, we have:

\begin{thm}
	Assume that $\ell\notin\{3,5\}$ is a prime.
	Let \(G \unlhd A\) be finite groups and consider a Sylow \(\ell\)-subgroup \(P\) of \(G\). Then there is an \(\N_A(P)\)-invariant injection
	\[\Omega : \IBr(\N_G(P)) \hookrightarrow \IBr(G)\]
	such that
	\[(A_\varphi, G, \varphi) \geq_c (\N_A(P)_\theta, \N_G(P), \theta)\]
	for every \(\theta \in \IBr(\N_G(P))\) and \(\varphi = \Omega(\theta)\).
\end{thm}

\begin{proof}
Follow by \cite[Thm.~2.13]{FMR26}, \cite[Thm.~5]{MNT23} and Corollary~\ref{cor:great7} directly.
\end{proof}

%%%%%%%%%%%%%%%%%%%%%%%%%%%%%%%%%%%%%%%%%%%%%%%%%%%%%%%%%%%%%%%%%%%%%%%%


\begin{thebibliography}{99}
		
		
\bibitem{Al87}
{\sc J.\,L.~Alperin},  Weights for finite groups. In: \emph{The Arcata Conference on Representations of Finite Groups, Arcata, Calif. (1986), Part I}. Proc. Sympos. Pure Math., vol.~{\bf 47}, Amer. Math. Soc., Providence, 1987, pp. 369--379.
		
\bibitem{AD12}
{\sc J. An, H. Dietrich}, The AWC-goodness and essential rank of sporadic simple groups. \emph{J. Algebra \bf 356} (2012), 325--354.
	
\bibitem{AHL24}
{\sc J. An, G. Hiss,  F. L\"ubeck}, The inductive blockwise Alperin weight condition for the Chevalley groups~$F_{\!4}(q)$. \emph{Mem. Amer. Math. Soc. \bf 304} (2024), no. 1530.
		
\bibitem{Br16}		
{\sc T. Breuer}, Constructing the ordinary character tables of some Atlas groups using character theoretic methods. arXiv:1604.00754v2.				
		
\bibitem{BM92}
{\sc M. Brou\'e, G. Malle}, Th\'eor\`emes de Sylow g\'en\'eriques pour les groupes r\'eductifs sur les corps finis. \emph{Math. Ann. \bf 292} (1992), 241--262.

\bibitem{BMM93}
{\sc M. Brou\'e, G. Malle, J. Michel}, Generic blocks of finite reductive groups. \emph{Ast\'erisque \bf 212} (1993), 7--92.
		
\bibitem{BS22}
{\sc J. Brough, B. Sp\"ath}, A criterion for the inductive Alperin weight  condition. \emph{Bull. London Math. Soc. \bf 54} (2022), 466--481.
		
\bibitem{Ca94}
{\sc M. Cabanes}, Unicit\'e du sous-groupe ab\'elien distingu\'e maximal dans certains sous-groupes de Sylow. \emph{C. R. Acad. Sci. Paris S\'er. I Math. \bf 318} (1994), 889--894. 		
		
\bibitem{CE94}
{\sc M. Cabanes, M. Enguehard}, On unipotent blocks and their ordinary characters. \emph{Invent. Math. \bf 117} (1994), 149--164.
		
\bibitem{CE04}
{\sc M. Cabanes, M. Enguehard}, \emph{Representation Theory of Finite Reductive Groups}. New Math. Monogr., {vol.~\bf 1}, Cambridge University Press, Cambridge, 2004.
		
\bibitem{CS13}
{\sc M. Cabanes, B. Sp\"ath}, Equivariance and extendibility in finite reductive groups with connected center. \emph{Math. Z. \bf 275} (2013), 689--713.
		
\bibitem{CS17a}
{\sc M. Cabanes, B. Sp\"ath}, Equivariant character correspondences and inductive McKay condition for type~$\mathsf A$. \emph{J. reine angew. Math. \bf 728} (2017), 153--194.
		
\bibitem{CS17b}
{\sc M. Cabanes, B. Sp\"ath}, Inductive McKay condition for finite simple groups of type~$\mathsf C$. \emph{Represent. Theory \bf 21} (2017), 61--81.
		
\bibitem{CS19}
{\sc M. Cabanes, B. Sp\"ath}, Descent equalities and the inductive McKay condition for types~$\ty B$ and~$\ty E$. \emph{Adv. Math. 356} (2019), 106820.
		
\bibitem{CS26}
{\sc M. Cabanes, B. Sp\"ath}, The McKay Conjecture on character degrees.  \emph{Ann. of Math. (2) \bf 203} (2026), 933--1032.
		
\bibitem{CCNPW85}
{\sc J.\,H.~Conway, R.\,T.~Curtis, S.\,P.~Norton, R.\,A.~Parker, R.\,A.~Wilson}, \emph{Atlas of Finite Groups: Maximal Subgroups and Ordinary Characters for Simple Groups}. Oxford University Press, Eynsham, 1985.

\bibitem{CR62}
{\sc C.\,W.~Curtis, I. Reiner}, \emph{Representation Theory of Finite Groups and Associative Algebras}. Pure Appl. Math., vol. {\bf 11}. Interscience Publishers, New York--London, 1962.

\bibitem{DM90}
{\sc F. Digne, J. Michel},  On Lusztig's parametrization of characters of finite groups of Lie type. \emph{Ast\'erisque \bf 181--182} (1990), 113--156.
				
\bibitem{F25}
{\sc Z. Feng}, Character triples and weights. \emph{J. Algebra \bf 684} (2025), 690--725.		
				
\bibitem{FLZ21} 
{\sc Z. Feng, C. Li, J. Zhang}, Equivariant correspondences and the inductive Alperin weight condition for type~$\ty A$. \emph{Trans. Amer. Math. Soc. \bf 374} (2021),  8365--8433.

\bibitem{FLZ22} 
{\sc Z. Feng, C. Li, J. Zhang}, Inductive blockwise Alperin weight condition for type~$\ty B$ and odd primes. \emph{J. Algebra \bf 604} (2022), 533--576.

\bibitem{FLZ23}
{\sc Z. Feng, Z. Li, J. Zhang}, Morita equivalences and the inductive blockwise Alperin weight condition for type~$\mathsf A$. \emph{Trans. Amer. Math. Soc. \bf 376} (2023), 6497--6520.

\bibitem{FLZ26}
{\sc Z. Feng, Z. Li, J. Zhang}, Decomposition matrix determinants, basic sets and inductive blockwise Alperin weight condition.  Submitted.

\bibitem{FM22}
{\sc Z. Feng, G. Malle}, The inductive blockwise Alperin weight condition for type~$\ty C$ and the prime~2.  \emph{J. Aust. Math. Soc. \bf 113} (2022), 1--20.
		
\bibitem{FMR26}
{\sc Z.	Feng, J. M. Mart\'{i}nez, D. Rossi},  Alperin's bound and normal Sylow subgroups. \emph{J. London Math. Soc. (2) \bf 113} (2026), Paper No. e70470.

\bibitem{FS23}
{\sc Z. Feng, B. Sp\"ath}, Unitriangular basic sets, Brauer characters and coprime actions. \emph{Represent. Theory \bf 27} (2023), 115--148.

\bibitem{FW26}
{\sc Z. Feng, L. Wu}, The Brauer $A(\infty)$ condition and Navarro's conjecture on Brauer characters under coprime actions. 	arXiv:2609.26441.		

\bibitem{FYZ23}
{\sc Z. Feng, J. Yu, J. Zhang}, Radical subgroups of finite reductive groups. 	arXiv:2401.00156.

	
\bibitem{Ge93}
{\sc M. Geck}, Basic sets of Brauer characters of finite groups of Lie type II.  {\em  J. London Math. Soc. (2)  \bf 47} (1993), 255--268.
		
\bibitem{GH91}
{\sc M. Geck, G. Hiss}, Basic sets of Brauer characters of finite groups of Lie type.  {\em  J. reine angew. Math. \bf 418} (1991), 173--188.
		
\bibitem{GM20}
{\sc M. Geck, G. Malle}, \emph{The Character Theory of Finite Groups of Lie Type: A Guided Tour}. Cambridge Stud. Adv. Math., vol. {\bf~187}, Cambridge University Press, Cambridge, 2020.



\bibitem{GLS98}
{\sc D.~Gorenstein, R.~Lyons, R.~Solomon}, \emph{The Classification of the Finite Simple Groups, Number 3}. Math. Surveys Monogr., {vol. \bf~40}, American Mathematical Society, Providence, RI, 1998.
		
\bibitem{GHKMW93}
{\sc R. Gow, B. Huppert, R. Kn\"orr, O. Manz, W. Willems}, \emph{Representation Theory in Arbitrary Characteristic}. Casa Editrice Dott. Antonio Milani (CEDAM), Padua, 1993.		
		
\bibitem{Is06}
{\sc I.\,M.~Isaacs}, \emph{Character Theory of Finite Groups}.  AMS Chelsea Publishing, Providence, RI, 2006.
		
\bibitem{Is08}
{\sc I.\,M.~Isaacs},  \emph{Finite Group Theory}. Grad. Stud. Math., vol. {\bf~92}, American Mathematical Society, Providence, RI, 2008.		
		
\bibitem{IMN07}
{\sc I.\,M.~Isaacs, G.~Malle, G.~Navarro}, A reduction theorem for the McKay conjecture. \emph{Invent. Math. \bf 170} (2007), 33--101.
			
\bibitem{JLPW95}
{\sc C.~Jansen, K.~Lux, R.~Parker, R.~Wilson}, \emph{An Atlas of Brauer Characters}. London Math. Soc. Monogr. New Series, {\bf~11}, Oxford Science Publications. The Clarendon Press, Oxford University Press, New York, 1995.

\bibitem{Ka92}		
{\sc G. Karpilovsky}, \emph{Group Representations, vol. 1. part B.,
Introduction to Group Representations and Characters}. North-Holland Mathematics Studies, {\bf 175}. North-Holland Publishing Co., Amsterdam, 1992.
		
\bibitem{KM15}
{\sc R. Kessar, G. Malle}, Lusztig induction and $\ell$-blocks of finite reductive groups. \emph{Pacific J. Math. \bf 279} (2015), 269--298.
		
\bibitem{KS16}
{\sc S. Koshitani, B. Sp\"ath}, The inductive Alperin--McKay and blockwise Alperin weight conditions for blocks with cyclic defect groups and odd primes. \emph{J. Group Theory \bf 19} (2016), 777--813.

\bibitem{Li21}
{\sc C. Li}, The inductive blockwise Alperin weight condition for $\mathrm{PSp}_{2n}(q)$ and odd primes. \emph{J. Algebra \bf 567} (2021), 582--612.

\bibitem{Lu77}
{\sc G. Lusztig}, Irreducible representations of finite classical groups. \emph{Invent Math \bf 43} (1977), 125--175.

\bibitem{Lu84}
{\sc G. Lusztig}, \emph{Characters of Reductive Groups Over a Finite Field}. Ann. of Math. Stud., {vol. \bf~107}, Princeton University Press, Princeton, NJ, 1984.
		
\bibitem{Lu88}
{\sc G. Lusztig}, On the representations of reductive groups with disconnected centre. \emph{Ast\'erisque \bf 168} (1988), 157--166.
		
\bibitem{Lu08}
{\sc G. Lusztig}, Irreducible representations of finite spin groups. \emph{Represent. Theory \bf 12} (2008), 1--36.
		
\bibitem{Ma07}
{\sc G. Malle}, Height 0 characters of finite groups of Lie type. \emph{Represent. Theory \bf 11} (2007), 192--220.
		
\bibitem{Ma08}
{\sc G. Malle}, The inductive McKay condition for simple groups not of Lie type. \emph{Comm. Algebra \bf 36} (2008), 455--463.
		
\bibitem{Ma14}
{\sc G. Malle}, On the inductive Alperin--McKay and Alperin weight conjecture for groups with abelian Sylow subgroups. \emph{J. Algebra \bf 397} (2014), 190--208.
		
\bibitem{MNT23}
{\sc G. Malle, G. Navarro, P. H. Tiep}, On Alperin's lower bound for the number of Brauer characters. \emph{Transform. Groups \bf 28} (2023), 1205--1220.
		
\bibitem{MS16}
{\sc G. Malle, B. Sp\"ath}, Characters of odd degree.  \emph{Ann. of Math. (2)  \bf 184}  (2016), 869--908.
		
\bibitem{MT11}
{\sc G. Malle, D. Testerman}, \emph{Linear Algebraic Groups and Finite Groups of Lie Type}. Cambridge Stud. Adv. Math., {vol. \bf~133}, Cambridge University Press, Cambridge, 2011.
		
		
		
\bibitem{NT89}
{\sc H. Nagao, Y. Tsushima}, \emph{Representations of Finite Groups}. Academic Press Inc., Boston, 1989.
		
\bibitem{Na98}
{\sc G. Navarro}, \emph{Characters and Blocks of Finite Groups}. London Mathematical Society Lecture Note Series, {vol. \bf~250}. Cambridge University Press, Cambridge, 1998.

\bibitem{Na18}
{\sc G. Navarro},  \emph{Character Theory and the McKay Conjecture}. Cambridge Stud. Adv. Math., vol. {\bf~175}, Cambridge University Press, Cambridge, 2018.
		
\bibitem{NS14}
{\sc G. Navarro, B. Sp\"ath},  On Brauer's height zero conjecture. \emph{J. Eur. Math. Soc.  \bf 16} (2014),  695--747.
		
\bibitem{NT11}
{\sc G. Navarro, P. H. Tiep},   A reduction theorem for the {A}lperin weight conjecture. \emph{Invent. Math. \bf 184} (2011), 529--565.
				
\bibitem{Ro22}
{\sc D. Rossi}, Character triple conjecture for $p$-solvable groups. \emph{J. Algebra \bf 595} (2022), 165--193.

\bibitem{Ro23}
{\sc D. Rossi}, The McKay conjecture and central isomorphic character triples. \emph{J. Algebra \bf 618} (2023),42--55.
		
\bibitem{Sc16}
{\sc E. Schulte}, The inductive blockwise Alperin weight condition for $G_2(q)$ and ${}^3D_4(q)$. \emph{J. Algebra \bf 466} (2016), 314--369.		
	
\bibitem{ST54}
{\sc G. C. Shephard, J. A. Todd}, Finite unitary reflection groups. \emph{Canad. J. Math. \bf 6} (1954), 274--304.	
	
\bibitem{Sp09}
{\sc B. Sp\"ath}, The McKay conjecture for exceptional groups and odd primes. \emph{Math. Z. \bf 261} (2009), 571--595.
	
\bibitem{Sp10a}
{\sc B. Sp\"ath}, Sylow $d$-tori of classical groups and the McKay conjecture.~I. \emph{J. Algebra \bf 323} (2010), 2469--2493.
	
\bibitem{Sp10b}
{\sc B. Sp\"ath}, Sylow $d$-tori of classical groups and the McKay conjecture.~II. \emph{J. Algebra \bf 323} (2010), 2494--2509.
	
\bibitem{Sp12}
{\sc B. Sp\"ath}, Inductive McKay condition in defining characteristic. \emph{Bull. Lond. Math. Soc. \bf 44} (2012), 426--438.
		

		
\bibitem{Sp17}
{\sc B. Sp\"ath}, Inductive conditions for counting conjectures via character triples. In: \emph{Representation Theory -- Current Trends and Perspectives}. EMS Ser. Congr. Rep., Eur. Math. Soc., Z\"urich, 2017, pp.~665--680.
		
\bibitem{Sp18}
{\sc B. Sp\"ath}, Reduction theorems for some global-local conjectures. In: \emph{Local Representation Theory and Simple Groups}. EMS Ser. Lect. Math., Eur. Math. Soc., Z\"urich, 2018, pp.~23--61.

\bibitem{Sp23}
{\sc B. Sp{\"a}th}, Extensions of characters in type $\ty D$ and the inductive McKay condition, I. \emph{Nagoya Math. J. \bf 252} (2023), 906--958. 
		
\bibitem{Sp25}
{\sc B. Sp\"ath}, Extensions of characters in type $\ty D$ and the inductive McKay condition, II. \emph{Invent. Math. \bf 242} (2025), 45--122.
		
\bibitem{SV16}
{\sc B. Sp\"ath, C. Vallejo Rodr\'{i}guez}, Brauer characters and coprime action. \emph{J. Algebra \bf 457} (2016), 276--311.
				
\bibitem{GAP}
{\sc The GAP Group}, GAP -- Groups, Algorithms, and Programming, Version 4.15.1; 2025. (http://www.gap-system.org)					

\bibitem{Wi18}
{\sc R.\,A.~Wilson}, Maximal subgroups of $^2\ty E_6(2)$ and its automorphism groups. arXiv:1801.08374.

				
\end{thebibliography}
\end{document}